\documentclass[11pt]{article}
\usepackage{latexsym}
\usepackage{amssymb}
\usepackage{graphicx}
\usepackage{enumerate}
\usepackage{amssymb}
\usepackage{amsmath}
\usepackage{color}
\usepackage{tikz}
\usepackage{hyperref}
\hypersetup{%hypertex=true,
	colorlinks=true,
	linkcolor=blue,
	anchorcolor=blue,
	citecolor=blue}
\usepackage{amsthm}
\usepackage{cleveref}
\crefformat{section}{\S#2#1#3} 
\crefformat{subsection}{\S#2#1#3}
\crefformat{subsubsection}{\S#2#1#3}
\newtheorem{thm}{Theorem}[section]
\newtheorem{lem}[thm]{Lemma}
\newtheorem{cor}[thm]{Corollary}
\newtheorem{prop}[thm]{Proposition}
\newtheorem{rem}[thm]{Remark}

\newtheorem{defn}[thm]{Definition}

\newcommand{\be}{\begin{equation}}
\newcommand{\ee}{\end{equation}}
\numberwithin{equation}{section}
\def\mR{{\mathcal{R}}}
\def\mO{{\mathcal{O}}}

\def\tmO{{\tilde{\mathcal{O}}}}
\def\dr{\frac{d}{dr}}
\def\Tr{{\mathrm{Tr}}}
\def\N{\mathcal{N}}

\def\a{{\alpha}}

\def\k{\frac{k}{2}}

\def\dvol{\mathrm{dvol}}

\def\h{{\hbar}}

\def\ve{\delta}
\newcommand{\RNum}[1]{\uppercase\expandafter{\romannumeral #1\relax}}

\def\l{\lambda}

\def\b{\beta}

\def\R{\mathbb{R}}

\def\k{\kappa}

\def\half{\frac{1}{2}}

\def\n{\frac{n}{2}}
\def\e{\mathcal{E}}

\def\ep{\epsilon}

\def\Rt{{\tilde{R}_{\h}}}

\def\t{\Theta}
\newcommand{\ba}{\begin{aligned}}
	\newcommand{\ea}{\end{aligned}}
\newcommand{\bs}{\begin{split}}
	\newcommand{\es}{\end{split}}

\begin{document}
	\def\o{\omega}
	\def\dvol{\mathrm{dvol }}
	\def\O{\Omega}
	\def\l{\lambda}
	\title{Weyl Law for Schrödinger Operators on Noncompact Manifolds, Heat Kernel, and Karamata-Hardy-Littlewood Theorem}
	\author{Xianzhe Dai\thanks{Department of Mathematics, UCSB, Santa Barbara CA 93106, dai@math.ucsb.edu. 
			}
		\and Junrong Yan\footnote{Department of Mathematics, Northeastern University, Boston, MA, 02215, j.yan@northeastern.edu}
	}
	\date{}                                           
	\maketitle
	\abstract{
Building on our earlier work on heat kernel asymptotics for Schrödinger-type operators on noncompact manifolds, we establish both the classical and semiclassical Weyl laws for Schrödinger operators of the form $\Delta + V$ and $\hbar^2 \Delta + V$ on complete noncompact manifolds. While the semiclassical law can be approached via localization, the classical Weyl law has remained widely expected but unproven in this generality. We impose a mild bounded integral oscillation condition on \( V \) in addition to the assumptions that \( V \) diverges at infinity and satisfies a doubling condition.  In this setting, our oscillation condition is sharp and strictly weaker than all previously known assumptions, even in the Euclidean case. 

A central novelty of our approach is an extended Karamata–Hardy–Littlewood Tauberian theorem, adapted to accommodate non-regularly varying spectral asymptotics in noncompact settings, together with its semiclassical analogue. These Tauberian tools allow us to derive both versions of Weyl’s law within a unified framework.
        }
      %  \tableofcontents
		\section{Introduction}
        
	In 1911, Weyl \cite{weyl1911asymptotische} established a fundamental asymptotic formula describing the distribution of large eigenvalues of the Dirichlet Laplacian on a bounded domain \(X \subset \mathbb{R}^n\):
\begin{equation} \label{classweyl}
\mathcal{N}(\lambda) \sim (2\pi)^{-n} \omega_n \lambda^{n/2} |X|\quad \text{as} \quad \lambda \to +\infty,
\end{equation}
where \(\mathcal{N}(\lambda)\) counts the eigenvalues of the (positive) Laplacian not exceeding \(\lambda\), \(\omega_n\) is the volume of the unit ball in \(\mathbb{R}^n\), and \(|X|\) denotes the volume of \(X\).

Known as Weyl's law, this formula reveals a profound link between the spectral characteristics of quantum systems and the geometry of their classical counterparts. Over the past century, it has been extended to a variety of geometric and analytic contexts via diverse methods; see, for instance, \cite{ivrii2016100, shubin2001pseudodifferential, chavel1984eigenvalues,WeylSurvey}  for a comprehensive overview.

One powerful tool for proving Weyl laws is the Karamata–Hardy–Littlewood (KHL) Tauberian theorem:
\begin{thm}[KHL Tauberian Theorem  {\cite{karamata1931neuer}}]  \label{karama}
Let \(\mu\) be an increasing function on \([0, \infty)\), and let \(\alpha > 0\). If
\be\label{classkara}
\int e^{-t\lambda} \, d\mu(\lambda) \sim t^{-\alpha} L(t), \quad \text{as } t \to 0^+,
\ee
for some slowly varying function \(L\), then
\[
\int_0^\lambda d\mu(r) \sim \frac{\lambda^\alpha L(\lambda^{-1})}{\Gamma(\alpha + 1)}, \quad \text{as } \lambda \to \infty.
\]
Here, \(L\) is said to be \textbf{slowly varying at $0$} if for all \(c > 0\),
\[
\lim_{t \to 0^+} \frac{L(ct)}{L(t)} = 1.
\]
A function of the form \(t^{-\alpha} L(t)\) is called \textbf{regularly varying} of index \(\alpha \in \mathbb{R}\).
\end{thm}

In particular, the classical Weyl law~\eqref{classweyl} on compact manifolds follows from the heat kernel expansion combined with this Tauberian theorem.

The situation for Schr\"odinger operators on noncompact manifolds, however, is much more subtle and challenging. While semiclassical Weyl's law could be obtained via localization methods \cite{braverman2025semi}, the classical one is considerably more delicate.  It is widely expected to hold under appropriate geometric and analytic conditions, but no proof is known for arbitrary noncompact manifolds, more than fifty years after the definitive work of Rosenbljum \cite{rozenbljum1974asymptotics}. Existing results focus primarily on $\mathbb{R}^n$~\cite{de1950asymptotic, rozenbljum1974asymptotics, feigin1976asymptotic, fleckinger1981estimate, levendorskiui1996spectral, tachizawa1992eigenvalue}, etc., or on manifolds with specific geometric structures at infinity, such as asymptotically Euclidean spaces, asymptotically hyperbolic spaces, or manifolds with cylindrical ends, etc.~\cite{bonthonneau2015weyl, moroianu2008weyl, coriasco2021weyl, MP-WEYL, chitour2024weyl}.

The Dirichlet–Neumann bracketing method, though powerful for proving the classical Weyl law on \( \mathbb{R}^n \), faces fundamental obstacles on general manifolds; see \cref{adv}. In this paper, we instead develop a new Tauberian theorem and its semiclassical analogue, tailored to the non-regularly varying spectral asymptotics arising in noncompact settings (see \cref{idea}). Combining this with the heat kernel expansion techniques, we obtain both classical and semiclassical Weyl laws for Schr\"odinger operators on complete noncompact manifolds with bounded geometry. In $\S$ \ref{removing}, we briefly outline how the bounded geometry assumptions can be relaxed, and how our argument can be extended to magnetic Schr\"odinger operators.
	 
\subsection{Notations and Assumptions}
		In this paper, we assume that all of our (Riemannian) manifolds have bounded geometry: 
		\begin{defn}\label{bounded-geometry}
			Let $(M,g)$ be a complete Riemannian manifold. $(M,g)$ is said to have bounded geometry, if the following
			conditions hold:
			\begin{enumerate}[(1)]
				\item The injectivity radius of $(M,g)$ is bounded below by some positive constant $\tau_0$.
				\item The norm of the curvature tensor and its first covariant derivative are uniformly bounded above by a constant $R_0 > 0$.
			\end{enumerate}
          For the Euclidean space $M = \R^n$, we set $\tau_0 = \sqrt{n}$.
		\end{defn}
		Given a complete Riemannian manifold $(M, g)$, let $\Delta$ denote the Laplace-Beltrami operator acting on $C^{\infty}(M)$. (Our sign convention for the Laplace operator is the one that makes $\Delta$ a positive operator.)
 %In local coordinates $x^1, \dots, x^n$, we have
		%\[
		%\Delta = -\frac{1}{\sqrt{\det g}} \frac{\partial}{\partial x^i} \left( \sqrt{\det g} \, g^{ij} \frac{\partial}{\partial x^j} \right),
		%\]
		%where $g = g_{ij} dx^i dx^j$, $(g^{ij}) = (g_{ij})^{-1}$, and $\det g = \det(g_{ij})$. 
        The corresponding Schrödinger operator on $(M, g)$ takes the form $\Delta + V(x)$, where $V(x) \in L_{\mathrm{loc}}^{\infty}(M)$ is the potential function.

We assume that
\be\label{cond-V-2}
\mathrm{ess\ lim}_{d(p, p_0) \to \infty} V(p) = \infty,
\ee
meaning that for every $L > 0$, there exists $R > 0$ such that
$$
V(p) \geq L \quad \text{for almost every } p \text{ with } d(p, p_0) \geq R.
$$
Here $d$ is the distance function induced by $g$ and $p_0$ is some fixed point.

It is well known that under these conditions, the operator $\Delta + V(x)$ is essentially self-adjoint (cf. \cite{rofe1970conditions,oleinik1994connection}; see also \cite{braverman2002essential} for Schr\"odinger-type operators acting on sections of vector bundles). Moreover, the spectrum of $\Delta + V(x)$ is discrete, and each eigenvalue has finite multiplicity.

		The main object of our study is the eigenvalue counting function (counted with multiplicity)
		$$\N(\l):=\#\{\tilde{\l}:\tilde{\l} \text{ is an eigenvalue of $\Delta+V$},\tilde{\l}<\l\}. $$ 
		Here for a finite set $A$, $\#A$ denotes the number of elements in $A.$
		
We introduce some assumptions on the growth and regularity of $ V $, similar to those in \cite{rozenbljum1974asymptotics}.

Let $V \in L^\infty_{\mathrm{loc}}(M)$ satisfy \eqref{cond-V-2}, and define
\begin{equation} \label{notation:sigma}
    \sigma(\lambda) = \Big| \big\{ x \in M : V(x) \leq \lambda \big\} \Big|,
\end{equation}
where $| \cdot |$ denotes the measure of a set induced by the metric $g$ on $M$.
		\begin{defn} \label{Def-doubling}
			We say $V$ satisfies the doubling condition, if there exists $C_V>0$, such that 
			\begin{equation}\label{double}
				\sigma(2\l)\leq C_V\sigma(\l)
			\end{equation}
			when $\l\geq\l_0$ for some $\l_0>0$. 
		\end{defn}

One consequence of the doubling condition is that, for any $t>0$, $$\int_Me^{-tV(x)}dx<\infty.$$ See Proposition \ref{Prop-V} for details.

\begin{defn} \label{Def-a-reg}
Let $V \in L^\infty_{\mathrm{loc}}(M)$.
 For some $\b\in[0,\half]$, we say $V$ is $\b$-regular if there exists a decreasing continuous function $v: \R %[1,\infty)
\mapsto (0,\infty)$ with $\lim_{t\to\infty}v(t)=0$, such that for any $x,y\in M$, whenever $d(x,y)<\tau_0$, we have 
		\begin{equation}\label{secondcond}
					|V(x)-V(y)|\leq d(x,y)^{2\b}\max\{|V(x)|^{1+\b},1\}v\big(V(x)\big).
				\end{equation}
This can be thought of as a quantified H\"older continuity condition for $V$.

We set \be\label{defn-Ra}\mR_\b:=\{V\in L_{\mathrm{loc}}^\infty(M): \text{ $V$  satisfies \eqref{cond-V-2}, \eqref{double} and \eqref{secondcond} }\}. \ee
 
		\end{defn}

\def\mO{{\mathcal{O}}}

The quantified H\"older regularity can in fact be significantly relaxed in an integral sense. To avoid introducing too much technicality in the introduction, we state this weaker condition in \cref{defn-beta-Oscillation}, where we also introduce a much larger class $\mathcal{O}_\beta$, for $\beta \in [0, \tfrac{1}{2}]$.

\subsection{Main Results}

We begin by extending the classical KHL Tauberian theorem (\Cref{karama}) and formulating a semiclassical version. In this extension (compare \eqref{classkara} and \eqref{newkara}),  the right-hand side involves an additional measure $d\nu$, which is adapted to capture the non-regularly varying spectral asymptotics that arise in noncompact settings (see \Cref{idea}).

		\begin{thm}\label{Kar}  
			Let $\mu$ and $\nu$ be increasing functions on $[0, \infty)$, and let $\alpha \in (0, \infty)$. Suppose that:  
			\begin{enumerate}[(1)]  
				\item For all $t > 0$, $e^{-tr} \in L^1([0, \infty), d\mu) \cap L^1([0, \infty), d\nu)$.  
				\item $\nu$ satisfies the doubling condition: there exists a constant $C_\nu > 0,s_0>0$ such that for all $s \geq s_0$,  
				$
				\nu(2s) \leq C_\nu \nu(s).
				$ 
			\end{enumerate}  
			Then 
			\begin{equation}  \label{newkara}
				\int e^{-t r} \, d\mu(r) \sim t^{-\alpha} \int e^{-tr} \, d\nu(r) \quad \text{as } t \to 0^+,  
			\end{equation}  
			implies  
			\begin{equation*}  
				\int_{0}^{\lambda} \, d\mu(r) \sim \frac{1}{\Gamma(\alpha + 1)} \int_{0}^{\lambda} (\lambda - r)^\alpha \, d\nu(r) \quad \text{as } \lambda \to \infty.  
			\end{equation*}  
		\end{thm}  
\begin{rem}
 It is easy to construct an increasing function $\nu$ satisfying the doubling condition above, but whose Laplace transform
$$
\int_0^\infty e^{-tr} \, d\nu(r)
$$
is not asymptotically regularly varying (see \Cref{non-regular}). Hence, our theorem strictly extends the classical KHL Tauberian theorem (\Cref{karama}). Moreover, the same argument implies that, under the same assumptions as in \Cref{Kar}, for a slowly varying function \( L \),
\begin{equation*}
    \int e^{-t r} \, d\mu(r) \sim t^{-\alpha} L(t) \int e^{-t r} \, d\nu(r) \quad \text{as } t \to 0^+,
\end{equation*}
implies
\begin{equation*}
    \int_0^\lambda d\mu(r) \sim \frac{L(\lambda^{-1})}{\Gamma(\alpha + 1)} \int_0^\lambda (\lambda - r)^\alpha \, d\nu(r) \quad \text{as } \lambda \to \infty.
\end{equation*} 
\end{rem}
		
		\begin{thm}\label{Kar1}  
			Let $\{\mu_\h\}_{\h \in (0,1]}$ be a family of increasing functions on $[0, \infty)$, $\nu$ an increasing function on $[0, \infty)$, and $\alpha \in [0, \infty)$. Assume that
			%\begin{enumerate}[(1)]
				 there exists $t_0 > 0$ such that for all $t \geq t_0$, $e^{-tr} \in \bigcap_{\h > 0} L^1([0, \infty), d\mu_\h) \cap L^1([0, \infty), d\nu)$.  
				%\item For any open subset $U \subset \mathbb{R}^+$, $\nu(U) > 0$. 
			%\end{enumerate}  
			
			If, for any $t \geq t_0$,  
			\begin{equation*}  
				\int e^{-tr} \, d\mu_\h(r) \sim (t\h^2)^{-\alpha} \int e^{-tr} \, d\nu(r) \quad \text{as } \h \to 0^+,  
			\end{equation*}  
			then, for any bounded open interval $I$,  
			\begin{equation*}  
				\h^{2\alpha} \int_I \, d\mu_\h(r) \sim \frac{1}{\Gamma(\alpha + 1)} \int_I r_+^{\alpha - 1} * d\nu(r) \quad \text{as } \h \to 0^+, 
			\end{equation*}  
			where $r_+^{\alpha - 1} * d\nu(r)$ is the Lebesgue–Stieltjes measure associated with the increasing functions below
		\[
		 \int_0^r\int_0^{s} (s - \tilde{s})^{\alpha - 1} \, d\nu(\tilde{s})  ds.
		\]
		\end{thm}  
\begin{rem}
   Notably, this theorem does not require the doubling condition on $\nu.$
\end{rem}

Consider the eigenvalue counting functions (counted with multiplicity)
				\[
				\N_\h(\lambda) := \#\{\tilde{\lambda} : \tilde{\lambda} \text{ is an eigenvalue of } \h^2\Delta + V, \tilde{\lambda} < \lambda\}\]
                and $$\N(\l):=\N_{\h=1}(\l).$$     
 To study Weyl's law, we first establish the following.
\begin{thm}\label{heatexpan}
Let $V \in \mR_\b$ (or more generally, $V \in \mathcal{O}_\beta$; see \Cref{beta-oscillation}).
Then, as $t \to 0$, 	\begin{align}\begin{split}\label{asymtr1}
				\ \ \ \  \Tr(e^{-t(\Delta + V)}) \sim \frac{1}{(4\pi t)^{n/2}} \int_M e^{-tV(x)} \, dx.
		\end{split}\end{align}
Moreover, if $\beta>0$, then for fixed $t >0$, as $\h \to 0$,  	\begin{align}\begin{split}\label{asymtr2}
				 \Tr(e^{-t(\h^2\Delta + V)})
				\sim \frac{1}{(4\pi t\h^2)^{n/2}} \int_M e^{-tV(x)} \, dx .
		\end{split}\end{align}
\end{thm}

Recall that $\sigma(\lambda) := \big|\{x \in M : V(x) \leq \lambda\}\big|$. As we will see in \eqref{intev}, \eqref{asymtr1} is equivalent to  
    \be \label{asymtr1.1}
   \int e^{-t\l}d\N(\l) \sim \frac{1}{(4\pi t)^{n/2}} \int e^{-tr} \, d\sigma(r),\mbox{ as $t\to0$.}
    \ee
Similarly,  \eqref{asymtr2} can be written equivalently as, for fixed $t>0$
        \be \label{asymtr2.1}
        \int e^{-t\lambda} \, d\mathcal{N}_\h(\lambda) \sim \frac{1}{(4\pi t\h^2)^{n/2}} \int e^{-tr} \, d\sigma(r),\mbox{ as $\h\to0$.}
        \ee
       By combining \Cref{Kar}, \Cref{Kar1}, and \Cref{heatexpan} with \eqref{asymtr1.1}, \eqref{asymtr2.1}, \eqref{intlv}, and \eqref{intxn1}, we obtain:
		\begin{thm}[Weyl's law]\label{weyl1}\label{mainthm}
	Let $V \in \mR_\beta$ (or more generally, $V \in \mathcal{O}_\beta$; see \Cref{beta-oscillation}).

			\begin{itemize}
				\item  Then
				\be\label{heatWeyl1}  
				\N(\l) \sim (2 \pi)^{-n} \omega_{n}\int_M(\l-V)_+^{\n}\dvol, \quad \l\to\infty. 
				\ee  
				
				\item Assume that $\beta>0$, then for any bounded open interval $ I \subset \mathbb{R}^+ $,
				% \[
				% \N_\h(I) := \#\{\tilde{\l} : \tilde{\l} \text{ is an eigenvalue of } \h^2\Delta + V, \tilde{\l} \in I\}  
				% \]  
				satisfies  
				\be\label{heatWeyl2}  
				\h^n\mathcal{N}_\h(I) \sim (2 \pi)^{-n} |\{(x,\xi)\in T^*M : |\xi|^2+V(x)\in I\}|, \quad \h\to0, 
				\ee  
where \[
				\N_\h(I) := \#\{\tilde{\l} \in I: \tilde{\l} \text{ is an eigenvalue of } \h^2\Delta + V\}  .
				\]  
			\end{itemize}
			\end{thm}
\begin{rem}\label{rem-RCD}
\begin{itemize}
 \item In $\S$\ref{classical}, we show that even for $\R^n$, our result extends all known versions of the classical Weyl law under the doubling condition. Removing the doubling condition, on the other hand, involves a different flavor of Tauberian-type theorems, which will be explored in a future project.
\item Under \eqref{cond-V-2} and the doubling condition \eqref{double}, the $\beta$-oscillation condition \eqref{beta-regularity0} is sharp. This is shown in \Cref{appendix-B}.
%\item Physically, this example shows that a very narrow potential well can support fewer bound states than its volume suggests. See \cref{appendix-B} for details.
\item Our weak regularity assumptions on $V$ suggest that the result could extend to lower regularity settings, such as RCD spaces, Ricci limit spaces, and others. 

  %  \item Unlike the traditional Weyl laws for compact manifolds, the eigenvalue counting function need not be of power growth. For example, for $M=\mathbb R^2$ with the standard Euclidean metric and $V(x)=e^{|x|^2}$, one finds, by Theorem \ref{mainthm}, $\N(\l)\sim 1/4\, \l \ln \l $. It could also have power growth with arbitrary high power, instead of half the dimension of the manifold. For an example, let $V(x)$ be a smooth function on $\mathbb R^n$ so that $V(x)=|x|^{\alpha}$ on $\{ |x|>1 \} $ for some $\alpha>0$. Then by Theorem \ref{mainthm}, $\N(\l)\sim C(n,\alpha)\, \l^{n/2+n/\alpha} $. Surprisingly, similar behavior could happen for compact Ricci limit spaces \cite{dai2023singular}.
\end{itemize}
		\end{rem}
		\subsection{Main ideas and outline of the proof}\label{idea}
		First, we outline the standard proof of the classical Weyl law \eqref{classweyl} using the heat kernel expansion and the Tauberian theorem. For further details, see  \cite[$\S 1.6$]{WeylSurvey}.
		
		%Let $\Delta$ denote the Laplace-Beltrami operator on 
        On a closed Riemannian manifold $(X, g)$, the classical heat kernel expansion implies that
		\be\label{heatclose}
		\Tr(e^{-t\Delta}) \sim (4\pi t)^{-\n}  \mathrm{Vol}(X), \quad t \to 0^+.
		\ee  
		
		In terms of the eigenvalue counting function $\N(\lambda)$ of $\Delta$, \eqref{heatclose} can be rewritten as  
		\be\label{heatclose1}
		\int e^{-t\lambda} \, d\N(\lambda)=\Tr(e^{-t\Delta}) \sim (4\pi t)^{-\n}  \mathrm{Vol}(X), \quad t \to 0^+.
		\ee  
		Thus, by applying the classical KHL Tauberian theorem (\Cref{karama}),   
\[
\N(\lambda) \sim (2\pi)^{-n} \omega_n \mathrm{Vol}(X) \lambda^n.
\]
		
		Now, let us consider our noncompact setting. Let $\Delta+V$ be the Schr\"odinger operator on a noncompact Riemannian manifold $(M, g)$. We will prove that  
		\be\label{heata}
		\Tr(e^{-t(\Delta + V)}) \sim (4\pi t)^{-\n}  \int_M e^{-tV(x)} \, dx, \quad t \to 0^+.
		\ee  
				
            As we will see, in terms of $\sigma(\lambda) := \big|\{x \in M : V(x) \leq \lambda\}\big|$, and the eigenvalue counting function $\N(\lambda)$  for $\Delta + V$, %by \eqref{intev}, 
            \eqref{heata} is equivalent to  
		\be\label{heata1}
		\int e^{-t\lambda} \, d\N(\lambda) \sim (4\pi t)^{-\n} \int e^{-t\lambda} \, d\sigma(\lambda), \quad t \to 0^+.
		\ee  
		
		Comparing \eqref{heata1} with \eqref{heatclose1}, the right-hand side involves an additional measure, $\sigma(\lambda)$. To address this, we extend \Cref{karama} to cases where both sides involve measures, i.e., Theorem \ref{Kar}.   Using this extended KHL Tauberian Theorem (Theorem \ref{Kar}), we derive Weyl's law in the noncompact setting.

		Finally, we note that our heat kernel approach can be extended to the semiclassical setting, following a similar line of reasoning as outlined above.

        Now the remaining tasks can be summarized as follows:  
		\begin{enumerate}[(1)]
			\item Establish Extended KHL Tauberian theorem, which is done in \Cref{karamata}.
			\item Prove \Cref{heatexpan}. This %constitutes the most technical and labor-intensive part of the paper and is addressed 
            is carried out in \Cref{expansion}.
		\end{enumerate}

\subsection{Outlook}

In this subsection, we outline two directions where our extended KHL Tauberian theorems and Wely's law may have further applications within and beyond classical geometric analysis.\\  \ \\
\textbf{Quantum geometry and local mirror symmetry.}
Our semiclassical Tauberian theorem provides analytic tools that may be relevant beyond the classical Weyl law. In particular, \cite{grassi2016topological, codesido2017spectral} propose an approach to local mirror symmetry based on the spectral theory of quantum curves. In the weak coupling regime \((\hbar \to 0)\), this perspective connects to the Nekrasov–Shatashvili (NS) limit of topological string theory \cite{nekrasov2010quantization}, whose mathematical foundation remains incomplete. Our results may help build a rigorous bridge in this setting. Furthermore, in some strong coupling 't Hooft limit \((\hbar \to \infty)\), this framework relates to Gromov–Witten theory, revealing deep ties between spectral theory, quantum curves, and enumerative geometry.\\
\textbf{Weyl-type laws on nonsmooth spaces.}
Our work is also motivated by the analysis of noncompact weighted manifolds \((M, g, e^{-f} \dvol_g)\), which naturally appear in Ricci solitons and Ricci limit spaces, and more generally in the study of metric measure spaces with synthetic Ricci curvature bounds, such as RCD spaces. In these contexts, the weighted Laplacian \(\Delta_f\) is spectrally equivalent to the Witten Laplacian, a Schr\"odinger operator. Due to the flexibility of our approach, which requires only mild regularity assumptions (See \Cref{beta-oscillation}), we expect it to yield new insights into spectral asymptotics in singular or nonsmooth settings. In particular, recent work such as \cite{dai2023singular} has revealed surprising phenomena regarding Weyl-type laws on RCD spaces, highlighting intriguing directions for further study.

\subsection{Dirichlet–Neumann Bracketing vs. Our Approach}\label{adv}

In $\R^n$, Dirichlet–Neumann bracketing is a classical and effective method for estimating the eigenvalue counting function of Schrödinger operators. This approach relies on the natural partition of \( \mathbb{R}^n \) into cubes. For each cube \( Q \subset \mathbb{R}^n \), one can explicitly compute the error
$
\big| \mathcal{N}(\lambda, \Delta) - \lambda^{n} |Q| \big|,
$
for all $\l>0$, where \(\mathcal{N}(\lambda, \Delta)\) counts the eigenvalues of the Laplacian on $Q$ with Dirichlet or Neumann boundary conditions and \( |Q| \) is the cube's volume.

On general noncompact Riemannian manifolds, however, such a natural decomposition does not exist. Even if the manifold is partitioned into domains with corners, uniformly controlling spectral errors for each \(\lambda\) on every domain remains highly challenging. While Seeley’s work \cite{seeley1980estimate} provides error estimates for Dirichlet problems on domains with corners, comparable control for Neumann boundary conditions is still unavailable.

As a result, Dirichlet–Neumann bracketing faces major obstacles when extending Weyl laws beyond $\R^n$. Our heat kernel method avoids these difficulties, providing a unified way to prove both classical and semiclassical Weyl laws under much weaker conditions.

Finally, in contrast to microlocal or symbolic calculus methods commonly used in the semiclassical setting, our approach requires weaker regularity assumptions.

\subsubsection*{Acknowledgments} The authors are deeply grateful to Maxim Braverman for his insightful comments and stimulating discussions. XD is partially supported by the Simons Foundation.
		
		\section{Extended KHL Tauberian Theorem}\label{karamata}
		
		In this subsection, we establish Extended KHL Tauberian theorems, Theorem \ref{Kar} and Theorem \ref{Kar1}. For readers' convenience, we restate them here in this section.

		\begin{thm}[Theorem \ref{Kar}] \label{Karp}
			Let $\mu$ and $\nu$ be increasing functions on $[0, \infty)$, and let $\alpha \in (0, \infty)$. Suppose the following conditions hold:  
			\begin{enumerate}[(1)]  
				\item For all $t > 0$, $e^{-tr} \in L^1([0, \infty), d\mu) \cap L^1([0, \infty), d\nu)$.  
				\item\label{thm-Kar-Item2} $\nu$ satisfies the doubling condition: there exists a constant $C_\nu,s_0 > 0$ such that for all $s \geq s_0$,  
				$
				\nu( 2s) \leq C_\nu \nu(s).
				$ 
			\end{enumerate}  
			Then  
			\begin{equation}\label{intaev}  
				\int e^{-t r} \, d\mu(r) \sim t^{-\alpha} \int e^{-tr} \, d\nu(r) \quad \text{as } t \to 0^+,  
			\end{equation}  
			implies  
			\[  
			\int_{0}^{\lambda} \, d\mu(r) \sim \frac{1}{\Gamma(\alpha + 1)} \int_{0}^{\lambda} (\lambda - r)^\alpha \, d\nu(r) \quad \text{as } \lambda \to \infty.  
			\]  
		\end{thm}

		Before proving \Cref{Karp}, we first establish the following lemmas. 
        Our first lemma shows that the integral \(\int_0^\infty e^{-tr} \, d\nu(r)\) for \(t > 0\) is {uniformly} controlled by \(\nu\left(\frac{b}{t}\right)\):
		\begin{lem}  \label{lem22}
			Let $\nu$ be a increasing function satisfying the conditions in \Cref{Karp}. For each $b > 0$, there exists a constant $C_{b,\nu} > 0$ such that for any $t\in(0,b/s_0)$,
			\[
			\int_{\frac{b}{t}}^\infty e^{-tr} \, d\nu(r) \leq C_{b,\nu} \int_0^{\frac{b}{t}} e^{-tr} \, d\nu(r).
			\]
			As a result, note that %e^{-b}\nu\left(\left[0,\frac{b}{t}\right]\right)\leq
            $\int_0^{\frac{b}{t}} e^{-tr} \, d\nu(r)\leq \nu\left(\frac{b}{t}\right),$ we have for any $t>0$,
			\[
			\int_0^\infty e^{-tr}d\nu(r)\leq (C_{b,\nu}+1)\nu\left(\frac{b}{t}\right)
			\]
		\end{lem}

		\begin{proof}
			To prove this, consider the intervals $I_k = \left[\frac{2^{k-1}b}{t}, \frac{2^k b}{t}\right)$. Then, we have:
			\begin{align*}
				&\ \ \ \ \int_{\frac{b}{t}}^\infty e^{-tr} \, d\nu(r)
				= \sum_{k=1}^\infty \int_{I_k} e^{-tr} \, d\nu(r) \leq \sum_{k=1}^\infty e^{-b2^{k-1}} \nu\left( \frac{2^k b}{t}\right) \\
				&\leq e^{-b} \nu\left( \frac{b}{t}\right) \sum_{k=1}^\infty e^b e^{-b2^{k-1}} C_\nu^k \leq \left(\int_0^{\frac{b}{t}} e^{-tr} \, d\nu(r)\right) \sum_{k=1}^\infty e^b e^{-b2^{k-1}} C_\nu^k.
			\end{align*}
			Setting $C_{b,\nu} = \sum_{k=1}^\infty e^b e^{-b2^{k-1}} C_\nu^k$ completes the proof.
		\end{proof}
		For any $\alpha > 0$, let $\nu^\alpha$ be the increasing function given by
		\be\label{conv}
		\nu^\alpha(s):=\frac{1}{\Gamma(\a)}(r_+^{\a-1} * \nu)(s) := \frac{1}{\Gamma(\alpha)} \int_0^s\int_0^r (r - \tilde{r})^{\alpha - 1} \, d\nu(\tilde{r}) \, dr.
		\ee
		 Now we show that the function \(\nu^\a\) satisfies the doubling conditions.
		\begin{lem} \label{lem23}
		Let $\nu$ be an increasing function satisfying the conditions in \Cref{Karp}. Then, for any $\alpha > 0$, the function $\nu^\alpha$ defines a measure whose Laplace transform satisfies
\begin{equation} \label{laptrans}
\int e^{-tr} \, d\nu^\alpha(r) = t^{-\alpha} \int e^{-tr} \, d\nu(r).
\end{equation}
Moreover, $\nu^\alpha$ satisfies the doubling condition, with the constant $C_\nu$ in item~(\ref{thm-Kar-Item2}) of \Cref{Karp} replaced by $C_{\nu^\alpha} := C_\nu^2 \cdot 2^{2\alpha}$ and $s_0$ replaced by $2s_0$.

		\end{lem}
		
		\begin{proof}
			The equality \eqref{laptrans} follows from the properties of the Laplace transform easily.
			
			To show that $\nu^\alpha$ satisfies the doubling condition, observe that:
			\begin{align}\begin{split}\label{valpha}
					&\ \ \ \ \Gamma(\alpha) \cdot \nu^\alpha(s) = \int_0^s \int_0^r (r - \tilde{r})^{\alpha - 1} \, d\nu(\tilde{r}) \, dr \\
					&= \int_0^s \int_{\tilde{r}}^s (r - \tilde{r})^{\alpha - 1} \, dr \, d\nu(\tilde{r}) \quad \text{(by Fubini's theorem)} \\
					&= \frac{1}{\alpha} \int_0^s (s - \tilde{r})^\alpha \, d\nu(\tilde{r}).
			\end{split}\end{align}
			Using this, note that for $s\geq 2s_0$ (recalling $\Gamma(\alpha + 1) = \alpha \Gamma(\alpha)$):
			\begin{align*}
				&\ \ \ \ \nu^\alpha(2s) = \frac{1}{\Gamma(\alpha + 1)} \int_0^{2s} (2s - r)^\alpha \, d\nu(r) \leq \frac{(2s)^\alpha \nu(2s)}{\Gamma(\alpha + 1)} \\
				&\leq \frac{C_\nu^2 2^{2\alpha} \left(\frac{s}{2}\right)^\alpha \nu\left( \frac{s}{2}\right)}{\Gamma(\alpha + 1)} \leq \frac{C_\nu^2 2^{2\alpha} \int_0^{\frac{s}{2}} (s - r)^\alpha \, d\nu(r)}{\Gamma(\alpha + 1)}\\
				&\leq \frac{C_\nu^2 2^{2\alpha} \int_0^s (s - r)^\alpha \, d\nu(r)}{\Gamma(\alpha + 1)} = C_\nu^2 2^{2\alpha} \nu^\alpha(s).
			\end{align*}
			This establishes the doubling condition.
		\end{proof}
It is important that the limits below converge \textbf{uniformly}.
		\begin{lem}\label{lim}
			Let $\nu$ be a increasing function satisfying the conditions in \Cref{Karp}. Then for any $\alpha > 0$ and $0\leq \epsilon <1$,
            \[
			1 \leq \frac{\nu^\alpha\big((1+\epsilon)\lambda\big)}{\nu^\alpha( \lambda)} \leq (1+\sqrt{\epsilon})^\alpha + c\sqrt{\epsilon}^\alpha (1-\epsilon)^{-\alpha}
			\]
            for some constant $c=c(\nu,\a).$
            Thus, the following limit holds \textbf{uniformly} for any $\l\geq 2s_0$:
			\[
			\lim_{\epsilon \to 0} \frac{\nu^\alpha\big((1 + \epsilon)\lambda\big)}{\nu^\alpha(\lambda)} = 1.
			\]
		\end{lem}
		
		\begin{proof}
			Using \eqref{valpha} (ignoring the constant factor $\Gamma(\alpha+1)$), we have:
			\begin{align*}
				&\ \ \ \ \nu^\alpha\big( (1+\epsilon)\lambda\big) = \int_0^{(1+\epsilon)\lambda} \big((1+\epsilon)\lambda - r\big)^\alpha \, d\nu(r) \\
				&\leq \int_0^{(1-\sqrt{\epsilon})\lambda} \big((1+\epsilon)\lambda - r\big)^\alpha \, d\nu(r) + \int_{(1-\sqrt{\epsilon})\lambda}^{(1+\epsilon)\lambda} \big((1+\epsilon)\lambda - r\big)^\alpha \, d\nu(r) \\
				&= J_1 + J_2.
			\end{align*}
			If $r \leq \lambda(1 - \sqrt{\epsilon})$, then $(1+\epsilon)\lambda - r \leq (1+\sqrt{\epsilon})(\lambda - r)$.
			Hence,
			\[
			J_1 \leq (1+\sqrt{\epsilon})^\alpha \int_0^\lambda (\lambda - r)^\alpha \, d\nu(r) = (1+\sqrt{\epsilon})^\alpha \nu^\alpha(\lambda).
			\]
			For $(1-\sqrt{\epsilon})\lambda \leq r$, we have
			$
			(1+\epsilon)\lambda - r \leq 2\sqrt{\epsilon}\lambda.
			$
			Thus,
			\begin{align*}
				&\ \ \ \ J_2 \leq (2\sqrt{\epsilon})^\alpha \lambda^\alpha \nu\big( (1+\epsilon)\lambda\big) \leq 4^{\alpha}C_{\nu}\sqrt{\epsilon}^\alpha \left(\frac{\lambda}{2}\right)^\alpha \nu\left( \frac{(1+\epsilon)\lambda}{2}\right) \\
				&\leq 4^{\alpha}C_{\nu}\sqrt{\epsilon}^\alpha (1-\epsilon)^{-\alpha}\int_0^{\frac{(1+\epsilon)\lambda}{2}} (\lambda - r)^\alpha \, d\nu(r) \\
                &\leq c\sqrt{\epsilon}^\alpha (1-\epsilon)^{-\alpha} \int_0^\lambda (\lambda - r)^\alpha \, d\nu(r) \ \ \ (c= 4^{\alpha}C_{\nu})\\
                & = c\sqrt{\epsilon}^\alpha (1-\epsilon)^{-\alpha}\nu^\alpha(\lambda).
			\end{align*}
			
			Combining the estimates for $J_1$ and $J_2$, we obtain:
			\[
			1 \leq \frac{\nu^\alpha\big((1+\epsilon)\lambda\big)}{\nu^\alpha(\lambda)} \leq (1+\sqrt{\epsilon})^\alpha + c\sqrt{\epsilon}^\alpha (1-\epsilon)^{-\alpha}.
			\]
			By the Squeeze Theorem, the lemma follows.
			
		\end{proof}
		
		Now we are ready to prove \Cref{Karp}.
		
		\begin{proof}[Proof of \Cref{Karp}]
Let $\nu^\alpha$ denote the measure defined in \eqref{conv}, where we identify an increasing function with its associated Lebesgue–Stieltjes measure.
Define the scaled measures on $\mathbb{R}^+$ by setting
			\[
			\mu_t(A) := \mu\left(t^{-1} A\right), \quad \nu^\alpha_t(A) := \nu^\alpha(t^{-1} A),
			\]
			for any set $A \subset \mathbb{R}^+$.
			
			For any Borel set $A$, let $\chi_A$ denote its indicator function. Then for $\omega = \nu^\alpha$ or $\mu$,
			\[
			\int \chi_A(r) \, d\omega_t(r) = \int \chi_A(tr) \, d\omega(r);
			\]
			and for any $s \geq 1$,
			\begin{equation}\label{intemu}
				\int e^{-sr} \, d\omega_t(r) = \int e^{-tsr} \, d\omega(r).
			\end{equation}
			Hence, by \eqref{intemu}, \eqref{intaev}, and \eqref{laptrans}, we have
			\begin{equation}\label{aes}
				\int e^{-sr} \, d\mu_t(r) \sim \int e^{-sr} \, d\nu^\alpha_t(r), \quad t \to 0.
			\end{equation}
			
			Consider the space
			\[
			\mathcal{B} := \operatorname{span} \{ g_s : \mathbb{R}^+ \to \mathbb{R}^+ \mid g_s(r) = e^{-sr}, \ s \in [1, \infty) \}.
			\]
			By \eqref{aes}, for all $h \in \mathcal{B}$,
			\begin{equation}\label{ah}
				\int h(r) \, d\mu_t(r) \sim \int h(r) \, d\nu^\alpha_t(r), \quad t \to 0.
			\end{equation}
			By the Stone-Weierstrass theorem, $\mathcal{B}$ is dense in
			\[
			C_0(\mathbb{R}^+) := \{ f \in C(\mathbb{R}^+) \mid \lim_{r \to \infty} f(r) = 0 \}.
			\]
			
			For $\tau \in \left(\frac{1}{2},1)\cup(1,\frac{3}{2}\right)$, let $\eta_\tau \in C_c(\mathbb{R}^+)$ satisfy
			\[\ba
			0 &\leq \eta_\tau \leq 1, \quad \eta_\tau|_{[0, \tau]} \equiv 1, \quad \eta_\tau|_{[\frac{1+\tau}{2}, \infty)} = 0, \quad\text{if $\tau\in\big(\half,1\big)$},\\
            0 &\leq \eta_\tau \leq 1, \quad \eta_\tau|_{[0, \frac{1+\tau}{2}]} \equiv 1, \quad \eta_\tau|_{[\tau, \infty)} = 0, \quad\text{if $\tau\in\big(1,\frac{3}{2}\big)$}.
			\ea\]
			 Since $\eta_\tau e^r \in C_c(\mathbb{R}^+)$, there exists a sequence $\{h_j\} \subset \mathcal{B}$ such that $h_j \to e^r \eta_\tau$ uniformly. By \eqref{ah}, for each $j$,
			\begin{equation}\label{Kareq2}
				\lim_{t \to 0} \frac{\int h_j(r) e^{-r} \, d\mu_t(r)}{\int h_j(r) e^{-r} \, d\nu^\alpha_t(r)} = 1.
			\end{equation}

			Next we claim that we can interchange $\lim _{j \rightarrow \infty}$ and $\lim _{t \rightarrow 0}$. Consequently,
			\be\label{calim}
			\int \eta_\tau(r) d \mu_{t}(r)\sim\int \eta_\tau(r)d\nu^\a_t(r),t\to0.
			\ee
			Now we prove the claim. By Lemma \ref{lem22} (Note that by \Cref{lem23}, $\nu^\a$ also satisfies the doubling condition), there exists {$t$-independent $C>0$} (setting $C=C_{\half,\nu^\a}$ in \Cref{lem22}), s.t.,
			\be\label{const1}
			C\int\eta_{\tau}(r)d\nu^\a_t(r)\geq \int e^{-r}d\nu^\a_t(r).
			\ee

			For each $\epsilon>0$, there exists $j_0$, such that if $j>j_0$, $|h_j-\eta_\tau e^r|<\epsilon$.
			So by \eqref{const1}, for each $j> j_0$,
			\begin{equation}\label{Kareq1}
				\frac{\int \eta_\tau(r)d \nu_t^\a(r)}{\int h_je^{-r}d \nu_t^\a(r)}\in\left(\frac{1}{1+C\epsilon},\frac{1}{1-C\epsilon}\right)
			\end{equation}
			
			As a result, for $j=j_0+1$,
			\begin{align*}
				\begin{split}
					&\ \ \ \ \limsup_{t\to0^+}\frac{\int \eta_\tau(r)d\mu_t(r)}{\int \eta_\tau(r)d \nu_t^\a(r)}\leq \limsup_{t\to0^+}\frac{\int h_je^{-r}d\mu_t(r)+\epsilon\int e^{-r}d\mu_t(r)}{\int \eta_\tau(r)d \nu_t^\a(r)}\\
					&\leq 1+C\epsilon +\limsup_{t\to0^+}\frac{\epsilon\int e^{-r}d \nu_t^\a(r)}{\int\eta_\tau(r)d \nu_t^\a(r)}\mbox{ (By \eqref{Kareq2} and \eqref{Kareq1})}\\
					&\leq 1+2C\epsilon\mbox{ (By \eqref{const1})}.
				\end{split}
			\end{align*}
			Similarly, we can show that
			\[
			\liminf_{t\to0^+}\frac{\int \eta_\tau(r)d\mu_t(r)}{\int \eta_\tau(r)d \nu_t^\a(r)}\geq 1-2C\epsilon.
			\]
			Letting $\epsilon\to0$, we prove the claim.
			
			Now for $\tau<1$,
			\begin{align}\begin{split}\label{eq30}
					&\ \ \ \ \liminf_{t\to0^+}\frac{\int\chi_{[0,1]}d\mu_t}{\int\chi_{[0,1]}d\nu^\a_t}\geq \liminf_{t\to0^+}\frac{\int\eta_\tau(t)d\mu_t}{\int\chi_{[0,1]}d\nu^\a_t}\\
					&= \liminf_{t\to0^+}\frac{\int\eta_\tau(t)d\nu^\a_t}{\int\chi_{[0,1]}d\nu^\a_t}\geq \liminf_{t\to0^+}\frac{\int\chi_{[0,\tau]}d\nu^\a_t}{\int\chi_{[0,1]}d\nu^\a_t}
			\end{split} \end{align}
			where the equality in the second line follows from the claim.
			
			By Lemma \ref{lim} and \eqref{eq30}, setting $\tau\to 1^-$, we obtain that
			\be\label{limsup}
			\liminf_{t\to0^+}\frac{\int\chi_{[0,1]}d\mu_t}{\int\chi_{[0,1]}d\nu^\a_t}\geq1.
			\ee
			
			Similarly, we have
			\be\label{liminf}
			\limsup_{t\to0^+}\frac{\int\chi_{[0,1]}d\mu_t}{\int\chi_{[0,1]}d\nu^\a_t}\leq1.
			\ee

			By \eqref{limsup} and \eqref{liminf}
			\begin{equation}\label{dva}\int_{0}^{\lambda} d \mu(r) \sim {\int_0^\lambda  d\nu^\alpha(r)}, \ \lambda \rightarrow \infty.\end{equation}

			Lastly, by \eqref{valpha},
			\begin{align}\begin{split}\label{dv}
					\int_0^\lambda d\nu^\alpha(r)=\frac{1}{\alpha\Gamma(\alpha)}\int_0^\lambda(\lambda -\tilde{r})^\alpha d\nu(\tilde{r}).
			\end{split}\end{align}
			
			By (\ref{dva}) and (\ref{dv}), the result follows.
		\end{proof}

		Similarly, a semi-classical analogue of the Tauberian-type theorem holds.
		\begin{thm}[Theorem \ref{Kar1}]\label{Karp1}  
			Let $\{\mu_\h\}_{\h \in (0,1]}$ be a family of increasing functions on $[0, \infty)$, $\nu$ an increasing function on $[0, \infty)$, and $\alpha \in [0, \infty)$. Assume that
			%\begin{enumerate}[(1)]
				 there exists $t_0 > 0$ such that for all $t \geq t_0$, $e^{-tr} \in \bigcap_{\h > 0} L^1([0, \infty), d\mu_\h) \cap L^1([0, \infty), d\nu)$.  
				%\item For any open subset $U \subset \mathbb{R}^+$, $\nu(U) > 0$. 
			%\end{enumerate}  

			If, for any $t > 0$,  
			\begin{equation}\label{intaev1}  
				\int e^{-tr} \, d\mu_\h(r) \sim (t\h^2)^{-\alpha} \int e^{-tr} \, d\nu(r) \quad \text{as } \h \to 0^+,  
			\end{equation}  
			then, for any bounded open interval $I$,  
			\[
			\h^{2\alpha} \int_I \, d\mu_\h(r) \sim \frac{1}{\Gamma(\alpha + 1)} \int_I r_+^{\alpha - 1} * d\nu(r) \quad \text{as } \h \to 0^+.  
			\]  
			
		\end{thm}

		\begin{proof}
			
			We adopt the notation introduced in the proof of \Cref{Karp}. We may as well assume that $t_0=\half.$
			
			Recall that $\nu^\alpha$ is defined in \eqref{conv}, and let $$\tilde{\mu}_\h = \h^{2\alpha} \mu_\h.$$ Then, by \eqref{intaev1}, for each $h\in\mathcal{B}$
			\be\label{eq37}
			\int h(r) \, d \tilde{\mu}_\h(r) \sim \int h(r) \, d \nu^\a(r), \quad \h \to 0.
			\ee
			
			By \eqref{eq37}, there exists $\h_0 > 0$ such that the measures $\{e^{-r} d \tilde{\mu}_\h(r)\}_{\h < \h_0}$ are uniformly bounded. That is, there exists a constant $C > 0$ independent of $\h$, such that for $\h < \h_0$,
			\[
			 \int e^{-r} d \tilde{\mu}_\h(r)  \leq C.
			\]
			Let $f\in C_c([0,\infty))$. Then we can find $\{h_j\}_{j=1}^\infty\in\mathcal{B}$, s.t. $h_j(r)\to f(r)e^r$ as $j\to\infty$ uniformly.
			
			Since $\{e^{-r} d \tilde{\mu}_\h(r)\}_{\h < \h_0}$ is uniformly bounded, we can interchange the limits $\lim_{j \to \infty}$ and $\lim_{\h \to 0^+}$, so we have
			\be
			\lim_{j\to\infty}\int h_j(r)e^{-r}d\tilde{\mu}_\h=\int f(r) d \tilde{\mu}_\h \sim \int f(r) d \nu^\alpha(r)=\lim_{j\to\infty}\int h_j(r)e^{-r}d\nu^\a,  \h \to 0^+.
			\ee
			Let $ I=(c,d) $ be a bounded open interval and, for any $|\tau|<\frac{d-c}{2}$, let $I_\tau=(c+\tau,d-\tau)$. Then we have
            \be
            \lim_{\tau\to 0}\frac{\int\chi_{I_\tau}(r)d\nu^\a(r)}{\int\chi_{I}(r)d\nu^\a(r)}=1.
            \ee
        By proceeding as in the proof of \eqref{dva}, we obtain
%\[\ba
%\liminf_{\h\to0}\frac{\int \chi_I(r) \, d \tilde{\mu}_\h(r)}{ \int \chi_I(r) \, d \nu^\a(r)}\geq \liminf_{\h\to0}\frac{\int \chi_I(r) \, d \tilde{\mu}_\h(r)}{ \int g_j(r)e^{-r} \, d \nu^\a(r)}
%\ea
%		\]
            
			\[
			\int \chi_I \, d \tilde{\mu}_\h(r) \sim \int \chi_I \, d \nu^\a(r), \quad \h \to 0.
			\]
			That is,
			\[
			\h^{2\alpha} \int_I d \mu_\h(r) \sim \frac{1}{\alpha \Gamma(\alpha)} \int_I r_+^{\alpha - 1} *d \nu(r), \quad \h \to 0.
			\]
		\end{proof}
		
		\section{Heat Kernel Expansion for $\beta$-Oscillation Functions}\label{expansion}  
		We now focus on establishing the heat trace asymptotics for Schr\"odinger operators. From this point onward, we may \textbf{assume without loss of generality that}\be\label{V-geq-one} V \geq 1 \text{ a.e.}\ee  
The following proposition summarizes useful identities and estimates involving \( V \) and \( \sigma \). As mentioned, \eqref{intev}, \eqref{intlv} motivate our Extended KHL Tauberian Theorem.

        \begin{prop} \label{Prop-V}
         Assume that $V$ satisfies the doubling condition. 
            \begin{enumerate}[(1)]
                \item For any $t>0$, $\int_Me^{-tV(x)}dx<\infty.$ 
     \item For any $t>0$, \be\label{intev} \int_M e^{-tV(x)} \, dx=\int_0^\infty e^{-tr} \, d\sigma(r).\ee
     \item For any $\l>0$, \be\label{intlv}\int_{M} (\lambda - V)_{+}^{n/2} \, dx=\int_0^\lambda (\lambda - r)^{n/2} \, d\sigma(r),\ee
     where for any real number $x$, $x_+:=\max\{x,0\}$.
     \item For any open interval $I\subset\R^+$,
     \begin{align}\label{intxn1}
				\int_I  (r_+^{\n-1}*d\sigma)(r)
				=\omega_n^{-1}|{\{(x,\xi)\in T^*M:|\xi|^2+V(x)\in I\}}|.
		\end{align}
        %where for any $\a>0$, $r_+^{\alpha - 1} * d\sigma(r)$ is the measure
	%	\[
		% \Big(\int_0^r (r - \tilde{r})^{\alpha - 1} \, d\sigma(\tilde{r}) \Big) dr,
		%\]
        %and $dr$ denotes the Lebesgue measure restricted to $[0, \infty)$.
        Here for a measurable subset $ A \subset T^*M $, $ |A| $ denotes its measure (with respect to the measure $ \dvol_{T^*M} $ induced by $ g $). Also, $ \omega_n $ is the volume of the unit ball in $ \mathbb{R}^n $.
            \end{enumerate}
            
        \end{prop}
        \begin{proof}
        We may as well assume that $\l_0$ in \Cref{Def-doubling} is $2.$
     For (1), we note that
        \begin{eqnarray*}
    \int_Me^{-tV(x)}dx & \leq & \sigma(2) + \sum_{k=1}^{\infty} e^{-t2^k}
    \Big|\big\{x\in M: 2^k \leq V(x)\leq 2^{k+1} \big\}\Big| 
    \\
    & \leq & \sigma(2) + \sum_{k=1}^{\infty} e^{-t2^k} C_V^{k+1} \sigma(2)
    <\infty. 
    \end{eqnarray*}
 Note that \( \sigma \) is a function of bounded variation. Moreover, by the doubling condition, we have \( \lim_{r \to \infty} e^{-tr} \sigma(r) = 0,t>0 \), and since \( V \geq 1 \) a.e., it follows that \( \sigma(0) = 0 \). Thus, by integration by parts and Fubini’s theorem, we obtain for (2):
\[
\begin{aligned}
&\quad\int_0^\infty e^{-tr} \, d\sigma(r) = \int_0^\infty t e^{-tr} \sigma(r) \, dr \\
&= \int_0^\infty t e^{-tr} \int_M \chi_{\{V \leq r\}} \, dx\,  dr = \int_M  \int_{V(x)}^\infty t e^{-tr} \, dr\, dx \\
&= \int_M  \int_{V(x)}^\infty -\frac{d}{dr} e^{-tr} \, dr\, dx = \int_M e^{-tV(x)} \, dx.
\end{aligned}
\]
Similarly, one can show that
        \[
				\ \ \ \ \int_0^\lambda (\lambda - r)^{n/2} \, d\sigma(r) = \int_{\{V(x)\leq\l\}} (\lambda - V)^{n/2} \, dx\]
                and\[\ba
			&\ \ \ \ 	\int_I  (r_+^{\n-1}*d\sigma)(r)=\int_I\int_0^r(r-\tilde{r})^{\n-1}d\sigma(\tilde{r})dr\\
				&=\int_I\int_{\{V(x)\leq r\}}(r-V)^{\n-1}dxdr=\omega_n^{-1}|{\{(x,\xi)\in T^*M:|\xi|^2+V(x)\in I\}}|.
		\ea\]
        \end{proof}

\subsection{$\beta$-Oscillation conditions}\label{defn-beta-Oscillation}

In this subsection, we introduce a condition which is much weaker than quantified H\"older continuity.
In \cite{rozenbljum1974asymptotics}, Rozenbljum introduces a class of functions $\mO_\beta'$ for $\b\in[0,\half)$, consisting of functions $ V \in L_{\mathrm{loc}}^\infty(\R^n)$ that satisfy \eqref{cond-V-2} and \eqref{double}, and the following two conditions:
\begin{enumerate}[(1)]
    \item  For all $y ,z\in \R^n$ with $|z| \in (0,1)$,
\be\label{beta-regularity2}
\int_{|x-y|\leq \sqrt{n},\, |x+z-y|\leq \sqrt{n}} |V(x) - V(x+z)|\,dx \leq \eta(|z|)|z|^{2\beta}(\max\{1,|V(y)|\})^{1+\beta},
\ee
where $ 0\leq\eta\in C\big([0,1)\big)$ with $\eta(0)=0$.

\item Moreover,  there exists $C_V'>0$ such that
    \be\label{beta-regularity}
    |V(x)|\leq C_V'\big(\max\{1,|V(y)|\}\big)
    \ee
    for almost every $x,y$ with $d(x,y)\leq \sqrt{n}$.
\end{enumerate}

In this paper, we consider a larger class of functions:
\begin{defn}\label{beta-oscillation}
Let $V \in L_{\mathrm{loc}}^\infty(M)$ and $\beta \in[ 0,\half]$. We say that $V$ satisfies the \emph{$\beta$-oscillation condition} if there exist continuous, positive increasing functions $\eta, \mu \in C\big([0,+\infty)\big)$ with
\[
\eta(0) = 0, \quad \mu(\lambda) \leq \tau_0 \lambda^{1/2}, \quad \text{and} \quad \lim_{\lambda \to \infty} \mu(\lambda) = \infty,
\]
such that for all sufficiently large $\lambda > 0$, and for all $r \in (0, \mu(\lambda)\lambda^{-1/2}]$, the following holds:
\begin{equation}\label{beta-regularity0}
\int_{\Omega_\lambda} \int_{S_r(x)} |V(x) - V(y)| \, \dvol_{S_r(x)}(y)\, dx 
\leq \eta(\lambda^{-1})\, r^{n + 2\beta - 1}\, \lambda^{1 + \beta}\, \sigma(\lambda),
\end{equation}
where $S_r(x) := \{ y \in M : d(x, y) = r \}$, $\Omega_\lambda := \{ x \in M : V(x) \leq \lambda \}$, and $\dvol_{S_r(x)}$ denotes the induced Riemannian measure on the geodesic sphere $S_r(x)$.

We set
\be\label{defn-Ob}
\mO_{\beta} := \left\{ V \in L_{\mathrm{loc}}^\infty(M) : \text{$V$  satisfies  \eqref{cond-V-2}, \eqref{double} and \eqref{beta-regularity0}} \right\}.
\ee
\end{defn}

In fact, it is also natural to consider the following function space. We denote by $\tmO_{\beta}$ the space of functions that satisfy conditions as those in $\mO_{\beta}'$, but with \eqref{beta-regularity2} replaced by the following: for any $r\in(0,\tau_0)$, 
\be\label{beta-regularity1}
 \int_{B_{\tau_0}(x)} \int_{ S_r(z)\cap B_{\tau_0}(x)}|V(z) - V(y)|\dvol_{S_r(z)}(y) dz \leq \eta(r)\, r^{n+2\beta-1} (\max\{1 ,V(x)\})^{1+\beta},
\ee
holds for any $x\in M$.
%\end{defn}

In \cref{appendix-A}, using the volume comparison and the Vitali covering lemma, we show that
$\tmO_{\beta}\subset \mO_{\beta}.$ We also prove that $\mO_\beta'$ is strictly contained in $ \mO_\beta$ when $M = \R^n$. See $\S$\ref{appendix a1}- $\S$\ref{appendix a4} for more examples of functions in $\mO_\b$.

		\subsection{Parametrix Construction}  
		
		In this subsection, we construct a parametrix \(k_\h^0\) for the heat kernel \(K_\h\) of the semi-classical Schrödinger operator \(\h^2\Delta + V\). Our approach follows the framework developed in \cite{DY2020index}, with additional insights from \cite{fan2011schr}. However, in this paper, we focus only on the leading-order term in the asymptotic expansion, and our method differs slightly from that in \cite{DY2020index}. 
		
		 Recall that ${\tau_0}>0$ is a injectivity radius lower bound of $M.$ Then for $d(x,y)<{\tau_0}$, set
		\begin{equation} \label{ehk}
			\e_0(t, x, y)=\e_{0,\h}(t,x,y)=\frac{1}{(4\pi \h^2t)^{\n}}\exp\left(-\frac{d^2(x,y)}{4\h^2t}\right),
		\end{equation}
		and
		\begin{equation} \label{cthk}
			\e_{1}(t,x,y)=\exp\big(-tV(x)\big).
		\end{equation} 
		
		A direct computation gives us the following formulas.% (the first two are well known). 
		
		\begin{lem}\label{e00}
			For \( y \in B_{\tau_0}(x) \), in normal coordinates centered at \( x \), we have:
\[
\begin{aligned}
\nabla \e_0 &= -\frac{\e_0}{2\h^2 t} \, r \nabla r, \quad &&(\partial_t + \h^2 \Delta) \e_0 = \frac{\e_0}{4tG} \, \nabla_{r \nabla r} G, \\
\nabla \e_1 &= 0, \quad &&\Delta \e_1 = 0.
\end{aligned}
\]

		%	and
		%	\[
		%	\nabla_{r \nabla r} H(x, y) + H(x, y) - V(y) = 0.
		%	\]
			Here, $G(x,y) := \det(g_{ij})$ is associated with the normal coordinates near $x$, and
			operators act on the $y$-component. 
		\end{lem}

		Let \be\label{k0} k_\h^0:=\e_{0,\h}\e_1G^{-1/4} \mbox{ and } R_\h:=\left(\partial_t+\h^2\Delta+V\right)k_\h^0,\ee
			where 
		operators act on the $y$-component. 
		 Then, 
		by a straightforward computation and using \Cref{e00}, we have
		\begin{prop} Near the diagonal $\{(x,x):x\in M\}\subset M\times M$,
			\begin{flalign}\label{rh}
				R_\h=\e_0\e_{1}\Big(
				\h^2\Delta G^{-1/4}%+\nabla_{r\nabla r}h_T\t_k 
				+\big(V(y)-V(x)\big)\Big).
			\end{flalign}
		
		\end{prop}
		\begin{proof}
			
			For any $u\in C^\infty(M\times M)$ and supported near the diagonal, we compute:
			\begin{flalign*} &\ \ \ \ \left(\partial_t+\h^2\Delta +V\right)(\e_0\e_{1}u) \\
				& = \left[\left(\partial_t+\h^2\Delta\right)\e_0\right]\e_1 u + 
				\left[\left(\partial_t +V\right)\e_1\right] \e_0 u  + %\e_0 \e_1 \left(\pat+\h^2\Delta\right)u 
\e_0 \e_1 \h^2\Delta u  
            -2\h^2 \langle \nabla \e_0, \nabla u \rangle \e_1 . \end{flalign*}
			Using Lemma \ref{e00}, we have 
			\begin{align*}
				\left[\left(\partial_t+\h^2\Delta\right)\e_0\right]\e_1 u =t^{-1} \e_0 \e_1  \frac{1}{4G}(\nabla_{r\nabla r}G)u ;\\
				\left[\left(\partial_t +V\right)\e_1\right] \e_0 u =\big(V(y)-V(x)\big) \e_0 \e_1 u ; \\
				%\e_0 \e_1 \left(\pat+\h^2\Delta\right)u =  \h^2\Delta u; \\
				-2\h^2 \langle \nabla \e_0, \nabla u \rangle \e_1 = t^{-1}\e_0 \e_1  \nabla_{r\nabla r}u.
			\end{align*}
			Note that $G^{-1/4}$ solves \[\nabla_{r\nabla r}u + \left(\frac{1}{4G}\nabla_{r\nabla r}G\right)u =0,\]
			our result follows.
		\end{proof}
		
		\def\nowwe{1}
		\if\nowwe0
		Now we can follow the standard procedure  to find suitable $\t_j=\t_{\h,j}$ , $j=0,1,...,k,$ such that 
		\begin{eqnarray} 
			\Big(\partial_t+\h^2\Delta+V\Big)(\e_0\e_1 u)  &  = & t^kR_{k,\h}(t,x,y), \label{parametrix} \\
			\t_{\h,0}(x,x) &  = & \operatorname{Id}  \label{initial}
		\end{eqnarray}
		where $R_{k,\h}(t,x,y)$ is $C^0$ in $t\in[0,\infty)$.
		This amounts  to solving ODEs inductively. In other words, we want to solve
		\begin{align*} 
			&\ \ \ \nabla_{r\nabla r}\t_{j+1} + \left(j+1+\frac{1}{4G}\nabla_{r\nabla r}G\right)\t_{j+1} + \h^2\Delta \t_j \\&- \h^2 \Delta H\t_{j-1}+2\h^2\nabla_{\nabla H} \t_{j-1} - \h^2|\nabla H|^2\t_{j-2} =0
		\end{align*}
		for $j=-1, 0, \cdots, k-1$. It can be rewritten as
		\begin{equation} \label{recursive}
			\nabla_{r\nabla r}(r^{j+1} G^{1/4}\t_{j+1})= \h^2 r^{j+1} G^{1/4} \left[ -\Delta \t_j
			+ \Delta H\t_{j-1}-2 \nabla_{\nabla H} \t_{j-1} + |\nabla H|^2\t_{j-2} \right].
		\end{equation}
		
		For $j=-1$, we have
		$\dr(G^{1/4 }\t_{\h,0})=0$. Together with the initial condition  $\t_{h,0}(x,x)=\operatorname{Id}$, one obtains
		$\t_{h,0}=G^{-1/4}\operatorname{Id}$.
		
		For $j=0$, we have
		$\dr(rG^{1/4 }\t_{\h,1})=-\h^2G^{1/4}\Delta \t_{\h,0}$; hence we can solve $\t_{\h, 1}$ explicitly in terms of $\t_{\h,0}$.
		
		Similarly, for $1\leq j\leq k-1$, $\t_{\h,j+1}$ can be solved recursively from the equation 
		\begin{flalign*}
			\dr(r^{j+1}G^{1/4}\t_{\h,j+1})&=-\h^2r^{j}G^{1/4}(\Delta\t_{\h,j}-\Delta H\t_{\h,j-1}\\
			&+2\nabla_{\nabla H} \t_{\h,j-1}-|\nabla H|^2\t_{\h,j-2}).
		\end{flalign*}
		With these choices for  $\t_{\h,j}$'s, we obtain (\ref{parametrix}), where
		\begin{flalign} \label{asyrem}
			R_{k, \h}&=\h^2\e_0\e_1\Big\{[\Delta \t_{\h,k}-\Delta H\t_{\h,k-1}+2\nabla_{\nabla H} \t_{\h,k-1}-|\nabla H|^2\t_{\h,k-2}]  \nonumber \\ 
			&+[-\Delta H\t_{\h,k}+2\nabla_{\nabla H} \t_{\h,k}-|\nabla H|^2\t_{\h,k-1}]t+[-|\nabla H|^2\t_{\h,k}]t^{2}\Big\}
		\end{flalign}
		
		For simplicity, we will abbreviate $\t_{\h,i}$ as $\t_{j}$ when we don't have to emphasis the dependence of $\h$.

		The following proposition follows from the above construction via an argument of induction, using the $a$-regular tame condition.
		\begin{prop}\label{asyexp}
			Each $\t_{\h,j}$ can be written as a polynomial of $\h$:
			\[\t_{\h,j}(x,y)=\sum_{l=[\frac{j+2}{3}]}^{j}\h^{2l}\t_{l,j}(x,y),\]
			where $\t_{l,j}$ is independent of $\h$, $[a]$ denotes the integral part of a real number $a$. 
			Moreover, for $a\in[0,\frac{1}{2})$
			\be\label{thetahj1} |\t_{\h,j}(x,y)|\leq C_j(\bar{V}_\gamma)^{\k j}\h^{2[\frac{j+2}{3}]}, \ee
			where $\k=\frac{2(1+a)}{3}$, $\bar{V}_{\gamma}=\sup_{p\in\gamma} |V(p)|$, $\gamma$ is the shortest geodesic connecting $x$ and $y.$
			
			When restricted on the diagonal of $M\times M$, $\t_{\h,j}(y,y)$ can be written as an algebraic combination the curvature of the metric $g$, the potential function $V$, as well as their derivatives, at $y$; in addition, $\t_{\h,0}(y,y)=\operatorname{Id}.$
		\end{prop}
		\fi
		\subsection{Remainder estimates on large bounded set}
        In this subsection, we establish the remainder estimate stated in \Cref{prop313}.

	Assume that $V \in \mathcal{O}_\beta$, and that \eqref{beta-regularity0} holds for some $\eta$ and $\mu$.
%We fix $\gamma_1\in(\gamma,\half).$
        
		 For any $T\gg 1$, we set
        \[V_T:=\max\{V,T\}.\] Let $\varphi\in C_c^\infty(\R)$ be a bump function  such that the support of $\varphi$ is contained in $[-1,1],$ $0\leq \varphi \leq 1$, and $\varphi|_{[-\half,\half]}\equiv1.$ 
		Let $\phi_T$ be given by:
		\begin{equation} \label{cutoff}
			\phi_T(x,y)=\varphi\left(\frac{d^2(x,y)V_T(x)}{\mu^2\big(V_T(x)\big)}\right).\end{equation}
		\begin{prop}\label{asym}
			
			Set  \[K_\h^{0,T}(t,x,y)=k_\h^0(t,x,y)\phi_T(x,y)\] 
			then 
			\begin{flalign*}
				\left(\partial_t+\h^2\Delta+V\right)K^{0,T}_{\h}(t,x,y)&=  \phi_T(x,y)R_{\h}(t,x,y)+\h^2\Delta\phi_T(x,y)k^0_{\h}(t,x,y)  \\
				&  -2\h^2 \big(\nabla\phi_T(x,y),\nabla k_{\h}^0(t,x,y)\big).
			\end{flalign*}
		\end{prop}
		Let $\tilde{R}^T_{\h}$ be given by:  
		\[
		\tilde{R}^T_{\h} = \phi_T(x,y)R_{\h}(t,x,y) + \h^2\Delta\phi_T(x,y)k^0_{\h}(t,x,y) - 2\h^2 \big(\nabla\phi_T(x,y), \nabla k_{\h}^0(t,x,y)\big),
		\]
		where derivatives are taken on the $y$-components.

		Note that the support of $\nabla\phi_T(x,y)$ and $\Delta\phi_T(x,y)$ lies outside the region $$\{y : d^2(x,y)V_T(x) < \mu^2\big(V_T(x)\big)
        \}
        .$$ Set $\chi_T(x, y)=1 $ if $d^2(x, y)V_T(x) < \mu^2\big(V_T(x)\big)$ and zero otherwise.  By \eqref{rh},
        \be\label{tilderhk1}  
		|\tilde{R}^T_{\h}| \leq C \hat\e_0 \e_1 \Big(\h^2V_T(x)\mu^{-2}\big(V_T(x)\big)+|V(y)-V(x)|\Big) \chi_T,
		\ee 
   where $\hat{\e}_0(t,x,y):=\e_0(2t,x,y)$.
   
	Let $K_{\h}$ denote the heat kernel of $\h^2\Delta+V$.
Let $K_M$ be the heat kernel of $\Delta$ on $M$.  It follows easily from the maximal principle and the standard heat kernel estimate that
	\begin{lem}\label{heates1}
    We have\begin{enumerate}[(1)]
        \item $0\leq K_\h(t,x,y)\leq K_M(t\h^2,x,y).$
        \item There exists positive constants $c_1$ and $c_2$, depending only on the bounded geometry data $({\tau_0},R_0)$, such that for $t\in(0,1]$,
		\be\label{heat estimate} 0\leq K_M(t,x,y)\leq c_1t^{-\n}\exp\left(-\frac{c_2d^2(x,y)}{t}\right).\ee 
    \end{enumerate}
		
	\end{lem}
	
\begin{lem}
Suppose $V$ satisfies the $\beta$-oscillation condition. Then there exists a constant $C$ depending only on $(n, \beta)$ such that for any sufficiently large $\lambda > 0$, any $r \in (0, \mu(\l)\l^{-\half}]$, and any $t \in (0,1)$,
\begin{equation}\label{beta-regularity00}
\int_{\Omega_\lambda} \int_{B_r(x)} e^{-\frac{d^2(x,y)}{4t}} |V(x) - V(y)| \, dy \, dx \leq C \eta(\l^{-1})\, t^{\n+\beta} \lambda^{1+\beta} \sigma(\lambda).
\end{equation}
\end{lem}

\begin{proof}
By \eqref{beta-regularity0} and Fubini's theorem, we get for $r \in (0, \mu(\l)\l^{-\half}]$:
\[
\begin{aligned}
&\quad\int_{\Omega_\lambda} \int_{B_r(x)}  e^{-\frac{d^2(x,y)}{4t}} |V(x) - V(y)| \, dy \, dx 
\\&= \int_0^r e^{-\frac{\rho^2}{4t}} \int_{\Omega_\lambda} \int_{S_\rho(x)} |V(x) - V(y)| \, \mathrm{dvol}_{S_\rho(x)}(y) \, dx \, d\rho \\
& \leq \eta(\l^{-1})\lambda^{1+\beta} \sigma(\lambda) \int_0^r e^{-\frac{\rho^2}{4t}} \rho^{n + 2\beta - 1} \, d\rho 
\leq C \eta(\l^{-1})\, t^{\n+\beta} \lambda^{1+\beta} \sigma(\lambda),
\end{aligned}
\]
where we used the standard estimate for the Gaussian-weighted integral of a power function.
\end{proof}

		\begin{prop}\label{prop313}Let $V\in O_\b$.
Then for any \(L > 1\), the following asymptotics hold:
\begin{equation}\label{eq56}
\int_{\{x \in M : V(x) \leq \frac{L}{t}\}} K_{\h=1}(t,x,x)\,dx \sim \frac{1}{(4\pi t)^\n} \int_{\{x \in M : V(x) \leq \frac{L}{t}\}} e^{-tV(x)}\,dx, \quad t \to 0,
\end{equation}
and if $\b>0$, for fixed $t>0$,
\begin{equation}\label{eq57}
\int_{\{x \in M : V(x) \leq \frac{L}{t}\}} K_{\h}(t,x,x)\,dx \sim \frac{1}{(4\pi t\h^2)^\n} \int_{\{x \in M : V(x) \leq \frac{L}{t}\}} e^{-tV(x)}\,dx, \quad \h \to 0.
\end{equation}
\end{prop}

\begin{proof} 
\def\te{{\tilde{\e}}}
Set $\tilde{c}_2 = \min\{c_2, 1/16\}$ and 
\begin{equation}\label{tilde-e}
\te(t,x,y) = (t \h^2)^{-\n} \exp\left( -\frac{\tilde{c}_2 d^2(x,y)}{t \h^2} \right),
\end{equation}
where $c_2$ is the constant in \Cref{heates1}.

For fixed small $t > 0$, we consider the parametrix $K_\h^{0,T}$ with 
\[
T =Lt^{-1}.
\]
Let
\be\label{gamma l t}
\gamma:=T^{-1}\mu^{2}(T)=(L^{-1}t)\mu^{2}(Lt^{-1}).
\ee
By Duhamel's principle, \eqref{tilderhk1}, and \Cref{heates1}, if $t$ is small enough,
\begin{equation}\label{Duhamel}
\begin{aligned}
&\quad \big|K_\h - K^{0,T}_\h\big|(t,x,x) = \left| \int_0^t \int_{B_{\gamma}(x)} K_\h(t-s, x, z) \Rt^T(s, x, z) \,dz\,ds \right| \\
&\leq \int_0^t \int_{B_{\gamma}(x)} \te(t-s, x, z) \te(s, x, z) 
\Big( \h^2 V_T(x)\mu^{-2}\big(V_T(x)\big) + |V(x) - V(z)| \Big) e^{-{s V(x)}} \,dz\,ds.
\end{aligned}
\end{equation}
Thus, by \eqref{gamma l t} and \eqref{Duhamel},
\begin{equation}\label{remainder-est1}
\begin{aligned}
&\quad \int_{\Omega_{Lt^{-1}}} \big|K_\h - K^{0,T}_\h\big|(t,x,x)\,dx \\
&\leq \int_0^t \int_{\Omega_{Lt^{-1}}} \int_{B_{\gamma}(x)} 
\te(t-s, x, z) \te(s, x, z) 
\left(  \h^2\gamma^{-1}+ |V(x) - V(z)| \right) e^{-s V(x)} \,dz\,dx\,ds \\
&=: I.
\end{aligned}
\end{equation}
Since $(M,g)$ has bounded geometry, it follows from volume comparison that there exists a constant $C_0$ such that for any $t \in (0,1]$,
\begin{equation}\label{Gaussint}
\frac{1}{(4\pi t)^\n} \int_{\{ y : d(x,y) < \tau_0 \}} e^{- \frac{d^2(x,y)}{t}} \,dy \leq C_0.
\end{equation}
Moreover, by a straightforward computation,
\begin{equation}\label{eq53}
\frac{d^2(x,z)}{t-s} + \frac{d^2(x,z)}{s} = \frac{t d^2(x,z)}{(t-s)s}.
\end{equation}
It follows from \eqref{beta-regularity00}, \eqref{tilde-e},  \eqref{Gaussint}, \eqref{eq53}, and the bound $e^{-l} \leq 1$ for $l > 0$ that
\begin{equation}\label{est-I1}
\begin{aligned}
&\quad I \leq C \int_0^t \left[ (t\h^2)^{-\n}   \h^2\gamma^{-1} + \h^{2\beta-n} 
\eta(tL^{-1}) (t-s)^\beta s^\beta t^{-\n - \beta} \left(Lt^{-1} \right)^{1+\beta} \right] \sigma(Lt^{-1}) \,ds \\
&\leq C'(L, \beta)(t\h^2)^{-\n} \Big( \h^2\gamma^{-1} t+ \h^{2\beta} \eta(t) \Big) \sigma(Lt^{-1}) \\
&\leq C''(L, \beta) e^L \Big( {\h^2}{\mu^{-2}(Lt^{-1})} + \h^{2\beta} \eta(t) \Big) 
\frac{1}{(4\pi t \h^2)^\n} \int_{\Omega_{Lt^{-1}}} e^{-t V(x)} \,dx.
\end{aligned}
\end{equation}

By \eqref{remainder-est1} and \eqref{est-I1}, fixing $\h = 1$ and letting $t \to 0$, we obtain \eqref{eq56}. The estimate in \eqref{eq57} can be established in a similar way.
\end{proof}
\def\final{1}
\if\final0
    Finally, we have the following remainder estimate:
	\begin{lem}\label{lem39}
	One obtains if $V$ is $a$-regular, there exists $(t,\h)$-independent constant $C>0$, such that 
		\[|K_\h-K_\h^0|(t,x,x)=|K_\h*\Rt|(t,x,x)\leq C(t\h^2)^{-\n}\Big[t\h^2+\h^{2a}v\big(V(x)\big)\cdot\big(1-e^{-7tV(x)/8}\big)\Big]. \]

	\end{lem}
	\begin{proof}
		\def\te{{\tilde{\e}}}
	 Set $\tilde{c}_2=\min\{c_2,7/32\}$ and \be\label{tilde-e}\te(t,x,y)= (t\h^2)^{-\n}\exp(-\frac{\tilde{c}_2d^2(x,y)}{t\h^2}),\ee where $c_2$ is the constant in \Cref{heates2}.
	 
	 Since \((M,g)\) has bounded geometry, it follows from volume comparison that there exists \(C_0\), such that for any \(t\in(0,1]\), 
	 \be\label{Gaussint}\frac{1}{(4\pi t)^\n}\int_{\{y:d(x,y)<{\tau_0}\}}e^{-\frac{\tilde{c}_2d^2(x,y)}{t}}dy\leq C_0.\ee  
	  Next, by straightforward computation,  
	 \be\label{eq53}  
	 \frac{d^2(x,z)}{t-s}+\frac{d^2(x,z)}{s}= \frac{td^2(x,z)}{(t-s)s}.  
	 \ee  
		Then
		\be\ba
	&\ \ \ \ 	|K_\h*\Rt|(t,x,x)\\
	&\leq C\int_0^t\int_{\{z:d(x,z)<{\tau_0}\}}\te(t-s,x,z)\te(s,x,z)\Big(\h^2+\h^{2a}v\big(V(x)\big)\cdot V(x)\Big)e^{-7sV(x)/8} dzds\\
    & \leq C\int_0^t\int_{\{z:d(x,z)<{\tau_0}\}}\te(t-s,x,z)\te(s,x,z)\Big(\h^2+\h^{2a}v\big(V(x)\big)\cdot V(x)e^{-7sV(x)/8}\Big) dzds\\
     & \leq C'(t\h^2)^{-\n}\int_0^t\Big(\h^2+\h^{2a}v\big(V(x)\big)\cdot V(x)e^{-7sV(x)/8}\Big) ds\\
		&\leq C''(t\h^2)^{-\n}\Big(t\h^2+\h^{2a}v\big(V(x)\big)\cdot\big(1-e^{-7tV(x)/8}\big)\Big),
		\ea\ee
		where the first inequality follows from \Cref{newlem}, Lemma \ref{heates1} and Lemma \ref{heates2}, and the second inequality follows from $e^{-7sV(x)/8}<1$ and the third inequality follows from \eqref{Gaussint} and \eqref{eq53}.

	\end{proof}
		
	\fi

	\def\secd{1}
	\if\secd0
		\section{Dealing with the Case of Lower Regularity}\label{low}
		\subsection{Dealing with}
		In this section, we relax the assumption that $ V $ is $ a $-regular. Instead, we consider $ V $ satisfying $ 0 \leq V \in C^1(M) $ and $ \lim_{x \to \infty} V(x) = +\infty $. Furthermore, there exist constants $ A_V > 0 $ and $ a \in [0, \frac{1}{2}) $ such that  
		\be\label{secondd11}  
		|\nabla V| \leq A_V(1 + V^{1+a}).  
		\ee  
		
		We will prove in \cref{provethm52} the following result:  
		
		\begin{thm}\label{approach}  
			For any $ \delta \in (0, \frac{1}{2}) $, there exists $ V_\delta \in C^\infty(M) $ such that:  
			\begin{enumerate}[(1)]  
				\item\label{i1} $ V_\delta $ is $ a $-regular for some $a\in(0,1)$.
				\item\label{i2} $ |V - V_\delta| \leq \delta(1 + V) $.  
				\item\label{i3} If $ V $ satisfies the doubling condition, then $ V_\delta $ also satisfies the doubling condition.  
			\end{enumerate}  
		\end{thm}

		\begin{subsubsection}{Proof of Theorem \ref{mainthm}}
			
			With Theorem \ref{approach} established, we proceed to prove Theorem \ref{mainthm}.  Using \eqref{intxn} and Lemma \ref{lim}, we first state the following lemma:  
			
			\begin{lem}\label{lim1}  
				There exists a sufficiently large constant $\lambda_0 > 0$ such that, for all $\lambda \geq \lambda_0$, the following limit holds uniformly as $\delta \to 0$:  
				\[
				\lim_{\delta \to 0} \frac{\int_{M} \big((1+\delta)\lambda - V(x)\big)_+^{\n} \, dx}{\int_{M} (\lambda - V(x))_+^\n \, dx} = 1.  
				\]  
			\end{lem}  
			
			Let $A(\delta):=\frac{1-2\delta}{1+2\delta}$ and $B(\delta):=A(\delta)^{-1}$. By Lemma \ref{lim1}, for any $\epsilon > 0$, there exists $\l$-independent $\delta_0 > 0$ such that for any $\delta < \delta_0$,  
			\be\label{lim2}  
			\left| \frac{\int_M \big(A(\delta) \lambda - V(x)\big)_+^\n \, dx}{\int_M \big(B(\delta) \lambda - V(x)\big)_+^\n \, dx} - 1 \right| < \epsilon.  
			\ee  
			
			Let $\mathcal{N}^\delta$ denote the eigenvalue counting function for $\Delta + V_\delta$.  
			Given that $(1+\delta)V + \delta \geq V_\delta \geq (1-\delta)V - \delta$, for any $\l>0$ big enough, we have
			\begin{align}
				\begin{split}\label{neweq1}
					&\ \ \ \ (\l-V(x))_+\leq \big((1+\delta)\l-(1+\delta)V(x)\big)_+\\&\leq \big((1+\delta)\l-V_\delta(x)+\delta\big)_+\leq \big((1+2\delta)\l-V_\delta(x)\big)_+.
			\end{split}\end{align}
			Similarly, we have for large enough $\l>0$,\be\label{neweq2}
			(\l-V(x))_+\geq \big((1-2\delta)\l-V_\delta(x)\big)_+.
			\ee
			Now for $\delta = \delta_0/2$ and ignoring the constant factor $(2\pi)^{-n} \omega_n$, we obtain:  
			\begin{align*}\begin{split}  
					& \ \ \ \ \liminf_{\lambda \to \infty} \frac{\mathcal{N}(\lambda)}{\int_M (\lambda - V(x))_+^\n \, dx}  
					\geq \liminf_{\lambda \to \infty} \frac{\mathcal{N}^\delta \big((1-2\delta)\lambda\big)}{\int_M \big((1+2\delta)\lambda - V_\delta(x)\big)_+^\n \, dx} \mbox{ (By \eqref{neweq1} and \eqref{neweq2})}\\  
					& \geq \liminf_{\lambda \to \infty} \frac{\mathcal{N}^\delta \big((1-2\delta)\lambda\big)}{\int_M \big((1-2\delta)\lambda - V_\delta(x)\big)_+^\n \, dx}  
					\cdot \frac{\int_M \big(A(\delta) \lambda - V(x)\big)_+^\n \, dx}{\int_M \big(B(\delta) \lambda - V(x)\big)_+^\n \, dx} \mbox{ (By \eqref{neweq1} and \eqref{neweq2})}\\  
					& \geq \liminf_{\lambda \to \infty} \frac{\mathcal{N}^\delta \big((1-2\delta)\lambda\big)}{\int_M \big((1-2\delta)\lambda - V_\delta(x)\big)_+^\n \, dx} (1-\epsilon) \quad \text{(By \eqref{lim2})} \\  
					& \geq (1-\epsilon),  
			\end{split}\end{align*}  
			where the final equality uses the fact that Weyl's law has already been established for $a$-regular potentials.  
			
			Similarly, we obtain:  
			\begin{align*}\begin{split}  
					\limsup_{\lambda \to \infty} \frac{\mathcal{N}(\lambda)}{\int_M \big(\lambda - V(x)\big)_+^\n \, dx} \leq (1+\epsilon).  
			\end{split}\end{align*}  
			
			Taking $\epsilon \to 0$, we conclude that Weyl's law holds for $V$ with lower regularity.  
			
			The semiclassical Weyl law for $V$ can be established in a similar manner. 
		\end{subsubsection}
		\subsubsection{Proof of Theorem \ref{approach} }\label{provethm52}
		We now turn to the proof of Theorem \ref{approach}.  
		
		Recall that $V$ satisfies the doubling condition if there exists $C_V > 0$ such that, for sufficiently large $\lambda$,  
		\[
		\sigma(2\lambda) \leq C_V \sigma(\lambda),  
		\]  
		where $\sigma(\lambda) := |\{x \in M : V(x) \leq \lambda\}|$.  
		
		From item \eqref{i2} in Theorem \ref{approach}, we can derive item \eqref{i3} as follows:  
		
		\begin{lem}  \label{dou}
			If $V$ satisfies the doubling condition with constant $C_V > 0$, and $V_\delta$ satisfies item \eqref{i2} in Theorem \ref{approach}, then $V_\delta$ also satisfies the doubling condition with constant $C_{V_\delta} \leq C_V^4$.  
		\end{lem}  
		\begin{proof}  
			Fix $\delta \in (0, \frac{1}{2})$. For any sufficiently large $\lambda > 1$, using $(1+\delta)V + \delta \geq V_\delta \geq (1-\delta)V - \delta$, we have:  
			\begin{align*}  
				&\ \ \ \ 	|\{x : V_\delta(x) \leq 2\lambda\}|  
				\leq \left| \left\{x : V(x) \leq \frac{2\lambda + \delta}{1-\delta} \right\} \right|  \leq |\{x : V(x) \leq 6\lambda\}| \\  & \leq C_V^4 |\{x : V(x) \leq 3\lambda/8\}|  \leq C_V^4 |\{x : V_\delta(x) \leq \lambda\}|.  
			\end{align*}  
		\end{proof}  
		
		It remains to construct $V_\delta$ that satisfies items \eqref{i1} and \eqref{i2} in Theorem \ref{approach}.
		
		Set $\Omega_0 := \{x \in M : V(x) < 2\}$. For $k \geq 1$, let $\Omega_k := \{x \in M : V(x) \in (2^{k-1}, 2^{k+1})\}$.  
		
		Set $Z_k := \{x \in M : V(x) = 2^k\}$. Since $\lim_{x\to\infty} V(x)=\infty$, $Z_k$ is compact. 
		Then $\partial \Omega_k = Z_{k-1} \cup Z_k$ for $k \geq 1$.
		\begin{lem}\label{diam}
			For $k \geq 1$, let $d_k$ denote the distance between $Z_{k-1}$ and $Z_k$. Then
			\[
			d_k \geq \frac{1}{2^{2+ka} A_V}.
			\]
		\end{lem}
		\begin{proof}
			Let $x \in Z_{k-1}$ and $y \in Z_k$ such that $d(x, y) = d_k$. Let $\gamma: [0, d_k] \to M$ be the shortest geodesic connecting $x$ and $y$.  
			
			For any $s \in (0, d_k)$, we have $V(\gamma(s)) \in (2^{k-1}, 2^k)$.  
			
			By the mean value theorem and \eqref{secondd11},  
			\[
			2^{k-1} = |V(x) - V(y)| \leq d_k \sup_s |\nabla V|(\gamma(s)) \leq 2 d_k A_V \sup_s |V|^{1+a} \leq 2 d_k 2^{k(1+a)}.
			\]
			The lemma follows.  
		\end{proof}

		Let $n = \dim(M)$. Let $\psi \in C_c^\infty(\R)$ such that $0 \leq \psi \leq 1$, $\psi|_{[-\half, \half]} \equiv 1$, and $\psi(s) = 0$ if $|s| \geq 1$.  Let $c_\phi:=\int_\R\psi(s) ds.$
		For any $\ve > 0$, let $\phi_\ve: M \times M \to \R$ be the function  
		\[
		\phi_\ve(x, y) = c_\psi^{-1}\o_n^{-1}\ve^{-n}\psi\left(\frac{d^2(x, y)}{\ve^2}\right).
		\]
		
		Since $M$ has bounded geometry, one can see that
		\be\label{unit}
		\int_M \phi_\ve(x,y) dy=1+O(\ve^2),\ve\to0.
		\ee

		Let $V_\ve^k \in C^\infty(\Omega_k)$ be given by  
		\[
		V_\ve^k(x) = \int_M \phi_{\frac{\ve}{2^{ka}}}(x, y) V(y) \, dy.
		\]  
		
		\begin{lem}\label{lem54}
			If $\ve > 0$ is sufficiently small, there exist constants $C > 0$ and $C_l > 0$ such that for any $k, l \in \mathbb{N}$ and $x \in \Omega_k$, the following hold:
			\begin{align}
				&|V_\ve^k(x) - V(x)| \leq (1 + C\ve)(V(x) + 1), \label{lem541} \\
				&\frac{|\nabla^l V_\ve^k(x)|}{1 + (V_\ve^k(x))^{1 + la}} \leq C_l \ve^{-l}. \label{lem542}
			\end{align}
		\end{lem}
		
		\begin{proof}
			By Lemma \ref{diam}, if $\ve$ is small and $x \in \Omega_k$, then for $d(x, z) \leq \frac{\ve}{2^{ka}}$, we have  
			\begin{equation}\label{eq64}
				|V(y)| \leq 2^{k+2} \leq 8(V(x) + 1).
			\end{equation}
			
			Then we compute:
			\begin{align*}
				\begin{split}
					&|V_\ve^k(x) - (1 + O(\ve^2))V(x)| = \left|\int_M \phi_{\frac{\ve}{2^{ka}}}(x, y) \big(V(y) - V(x)\big) \, dy \right| \\
					&\leq \frac{\ve}{2^{ka}} \int_M \phi_{\frac{\ve}{2^{ka}}}(x, y) \sup_{z : d(x, z) \leq \frac{\ve}{2^{ka}}} |\nabla V(z)| \, dy \\
					&\leq C\ve \int_M \phi_{\frac{\ve}{2^{ka}}}(x, y)(V(x) + 1) \, dy \quad \text{(by \eqref{eq64} and \eqref{secondd11})} \\
					&\leq C'\ve (V(x) + 1) \quad \text{(by \eqref{unit})}.
				\end{split}
			\end{align*}
			Thus, \eqref{lem541} is proved.
			
			Next, by the construction of $V_\ve^k$, we estimate:
			\begin{align*}
				\begin{split}
					|\nabla^l V_\ve^k(x)| &\leq C_l \ve^{-n} 2^{kna} \int_{\{y : d(y, x) \leq \frac{\ve}{2^{ka}}\}} \ve^{-l} 2^{kal} |V(y)| \, dy \\
					&\leq C_l' \ve^{-n} 2^{kna} \int_{\{y : d(y, x) \leq \frac{\ve}{2^{ka}}\}} \ve^{-l} (V(x) + 1)^{1 + al} \, dy \quad \text{(by \eqref{eq64})} \\
					&\leq C_l'' \ve^{-l} \big(1 + V(x)^{1 + al}\big) \\
					&\leq C_l''' \ve^{-l} \big(1 + (V_\ve^k(x))^{1 + al}\big) \quad \text{(by \eqref{lem541})}.
				\end{split}
			\end{align*}
			Thus, \eqref{lem542} is proved.
		\end{proof}
		
		Let $\{\eta_k\}_{k=0}^\infty$ be a partition of unity subordinate to the open cover $\{\Omega_k\}_{k=1}^\infty$. By Lemma \ref{diam}, we can ensure that  
		\begin{equation}\label{eq69}
			|\nabla \eta_k|(x) \leq C 2^{ka} \leq 2C (1 + V(x))^a.
		\end{equation}
		
		Set $V_\ve := \sum_{k=0}^\infty \eta_k V_\ve^k$. Then, item \eqref{i1} in \Cref{approach} follows from \eqref{lem542} and \eqref{eq69}, while item \eqref{i2} in \Cref{approach} follows from \eqref{lem541}.
		\fi
		\def\subsec{1}
		\if\subsec0
		\subsection{Dealing}
		In this subsection, we assume that $V$ satisfies the following:\\
		There exist $\alpha\in[0,\frac{1}{2}]$, and a continuous decreasing function $v:[0,\infty)\mapsto \R$ with $\lim_{t\to\infty}v(t)=0$, such that for any $x,y\in M$, whenever $d(x,y)<{\tau_0}$, we have
		\begin{equation}\label{secondcond}
			|V(x)-V(y)|\leq d(x,y)^{2\alpha}\max\big\{1,|V(x)|^{1+\alpha}\big\}v\big(V(x)\big).
		\end{equation}
		
		\begin{thm}\label{approach1}
			For any $ \delta \in (0, \frac{1}{2}) $, there exists $ V_\delta \in C^\infty(M) $ such that:  
			\begin{enumerate}[(1)]  
				\item\label{i12} $ |V - V_\delta| \leq \delta(1 + V) $.  
				\item\label{i13} If $ V $ satisfies the doubling condition, then $ V_\delta $ also satisfies the doubling condition.  
				\item For any $x,y\in M$, whenever $d(x,y)<{\tau_0}$, we have
				\begin{equation}\label{secondd2}
					|V_\delta(x)-V_{\delta}(y)|\leq 2d(x,y)^{2\alpha}\max\big\{1,|V_\delta(x)|^{1+\alpha}\big\}\cdot v\big((1-2\delta)V_\delta(x)\big)+2\delta(1+V_\delta(x)).
				\end{equation} 
				\item 	\label{i14} $\lim_{x\to\infty}\frac{|\nabla^kV|}{|V|^{1+\frac{k}{2}}} =0.$
			\end{enumerate}  
		\end{thm}
		
		\begin{proof}
			In fact, it is straightforward to see that item \eqref{i13} follows from item \eqref{i12}, and \eqref{secondd2} directly follows from item \eqref{i12} and \eqref{secondcond}.
			
			We use the same notation as in \cref{provethm52}. Let $v_k=v(2^{k-1}),k\geq 1$ and $v_0=v(0)$, then as in \Cref{diam}, by \eqref{secondcond}, one can show that
			\begin{equation}\label{eq70}
				d_k\geq \frac{1}{(2v_k)^{\frac{1}{2a}}2^{\frac{k}{2}}}.
			\end{equation}
			\def\tve{{\tilde{\ve}}}
			We construct $V_\ve^k \in C^\infty(\Omega_k)$ as follows:
			\[
			V_\ve^k(x) = \int_M \phi_{\ve_k2^{-k/2}}(x, y) V(y) \, dy,
			\]  
			where $\ve_k=\min\left\{\left(\frac{\delta}{v_k}\right)^{\frac{1}{2a}},2^{k/4}\right\}.$
			Proceeding as in the proof of \Cref{lem54}, using \eqref{secondcond}, one can see that for some constant $C$ and $C_l$
			\begin{equation}\label{eq72}
				|V_\ve^k(x)-V(x)|\leq (1+C\delta)(V(x)+1),x\in\Omega_k
			\end{equation}
			and
			\begin{equation}
				\frac{|\nabla^lV_\ve^k(x)|}{1+\big((V_\delta^k(x))\big)^{1+l/2}}\leq C_l\delta_k^{-l}\leq \max\left\{\left(\frac{v_k}{\delta}\right)^{\frac{1}{2a}},2^{-k/2}\right\},x\in\Omega_k.
			\end{equation}
			
			Let $\eta_k$ be the partition of unit w.r.t. $\{\Omega_k\}_{k=0}^\infty$, by \eqref{eq70}, we can assume that 
			\begin{equation}\label{eq74}
				|\nabla \eta_k|\leq (2v_k)^{\frac{1}{2a}}2^{\frac{k}{2}}.
			\end{equation}
			Then by \eqref{eq72}-\eqref{eq74}, one can show that $V_\delta=\sum_{k=0}^\infty \eta_kV_\delta^k$ satisfies our assumption as in \cref{provethm52}.
		\end{proof}
		\begin{thm}\label{limitthm1}
			The following limits hold uniformly for \(t, \h \in (0, 1]\):  
			\be\label{eq75}
			\lim_{\delta \to 0} \frac{\Tr(e^{-t(\h^2\Delta + V_\delta)})}{\Tr(e^{-t(\h^2\Delta + V)})} = 1,
			\ee  
			and  
			\be\label{eq76}
			\lim_{\delta \to 0} \frac{\int_M e^{-tV_\delta(x)} dx}{\int_M e^{-tV(x)} dx} = 1.
			\ee
		\end{thm}
		\fi
		\subsection{Proof of Theorem \ref{heatexpan}}\label{proofthm44}
Below is an outline of the proof of \Cref{heatexpan}.
\Cref{prop313} provides the asymptotic formulas for the integral of the heat kernel over a time-dependent bounded region (up to sets of measure zero).
Thus, to prove \Cref{heatexpan}, more specifically, \eqref{asymtr1} and \eqref{asymtr2}, we need to control the integrals of the heat kernel and the exponentiated potential \( e^{-tV(x)}\) outside the time-dependent region. The estimates in \Cref{prop311}, \Cref{prop390}, \Cref{prop39}, and \Cref{prop47} address this issue. There is a price to pay however, namely we have to sacrifice some time for the integral of the exponentiated potential, Cf. \Cref{prop47}. Thus, we will also need to show that it will not cause any problem for our final asymptotic formulas. This is dealt with using the uniform limit in \Cref{cor31}.

 %We show that the integral of \( e^{-tV(x)} \) is concentrated on the set \(\{x \in M : V(x) \leq \frac{A}{t}\}\) for some \( A > 0 \), which shows that the limit \eqref{eq50} converges uniformly.
We now look at the integral of \( e^{-tV(x)}\) outside a time-dependent region.
		\begin{prop}\label{prop311}Assume $V$ satisfies the doubling condition \eqref{double}. Then
		for any $\ep>0$, there exists $A = A(\epsilon)$ such that for all $t\in(0,2]$,
		\be\label{eq86}
		{\int_{\{x \in M : V(x) \geq \frac{A}{t}\}} e^{-tV(x)}\,dx}\leq \epsilon \int_M e^{-tV(x)}\,dx, 
		\ee
		\end{prop}
	\begin{proof}

		We may as well assume that $\l_0$ in \Cref{Def-doubling} is $2.$
		Then for any $A>8$,
		\begin{align}\begin{split}\label{growth2}  
				&\ \ \ \ \int_{\{x \in M : V(x) \geq \frac{A}{t}\}} e^{-{tV(x)}} \, dx 
				=  \sum_{k=1}^\infty \int_{\{x : 2^{k-1}At^{-1} \leq V(x) \leq 2^kAt^{-1}\}} e^{-{tV(x)}} \, dx \\  
				&\leq \sum_{k=1}^\infty e^{-2^{k-1}A} \sigma(2^kAt^{-1}) \leq e^{-A/2}\sigma(At^{-1}/2)\sum_{k=1}^\infty e^{-2^{k-2}A}C_V^{k+1}  \\  
				&\leq e^{-A/2}\left(\sum_{k=1}^\infty e^{-(2^{k-2}-2^{-1})A}C_V^{k+1}\right)\int_{\{x:V(x)\leq \frac{A}{2t}\}}e^{-tV(x)}dx \leq Ce^{-A/2}\int_{M}e^{-tV(x)}dx .  
		\end{split}\end{align}   
	\end{proof}
As alluded above, it is critical that the following limit \eqref{eq50} converges \textbf{uniformly}.
	\begin{cor}\label{cor31}
		Assume $V$ satisfies the doubling condition \eqref{double}, then the following limit holds \textbf{uniformly}:
		\be\label{eq50}
		\lim_{\delta\to0}	\frac{\int_{M} e^{-tV(x)}\,dx}{\int_M e^{-t(1-\delta)V(x)}\,dx} =1, \ t\in(0,1]
		\ee
	\end{cor}
\begin{proof}
	By \eqref{eq86},  for any $\ep>0$, there exists $A=A(\ep)$, such that for any $\delta\in(-\half,\half),t\in(0,1]$,
	\be\label{eq51}
		{\int_{\{x \in M : V(x) \geq \frac{A}{t}\}} e^{-t(1+\delta)V(x)}\,dx} \leq \epsilon {\int_M e^{-t(1+\delta)V(x)}\,dx}.
	\ee
	Next, for any $\delta\in(-\half,\half)$,
	\be\label{eq521}
	\int_{\{x \in M : V(x) \leq \frac{A}{t}\}} |e^{-tV(x)}-e^{-t(1-\delta)V(x)}|dx\leq (e^{|\delta| A}-e^{-|\delta|A})\int_M e^{-t(1-\delta)V(x)}dx.
	\ee
	%Then \eqref{eq50} follows from \eqref{eq51} and \eqref{eq521} easily.
    Let $$F(\delta)=\frac{\int_{M} e^{-tV(x)}\,dx}{\int_M e^{-t(1-\delta)V(x)}\,dx}.$$ 
    Writing the top integral into two parts corresponding to the region $V \geq A/t$ and $V\leq A/t$ and using \eqref{eq51}, we deduce
    $$ F(\delta) \leq \ep F(\delta) + \frac{\int_{V\leq A/t} e^{-tV(x)}\,dx}{\int_M e^{-t(1-\delta)V(x)}\,dx}. $$
    On the other hand, \eqref{eq521} yields
    $$\frac{\int_{V\leq A/t} e^{-tV(x)}\,dx}{\int_M e^{-t(1-\delta)V(x)}\,dx} \leq 1 + (e^{|\delta| A}-e^{-|\delta|A}). $$
    Combining, we conclude
    $$F(\delta) \leq \frac{1 + (e^{|\delta| A}-e^{-|\delta|A})}{1-\ep}. $$
    An easier argument gives us 
    $$F(\delta) \geq 1-(e^{|\delta| A}-e^{-|\delta|A}).  $$
    Our result follows.
\end{proof}

Next, we deal with the integral of heat kernel outside the time dependent region. For an eigenform \( u \) corresponding to an eigenvalue \( \leq \lambda \), we show that its \( L^2 \)-norm is concentrated on the set \(\{x \in M : V(x) \leq C\lambda\}\) for large \( C > 1 \).  

\begin{lem}\label{prop390}  
If \( u \) is an eigenform of \( \h^2\Delta + V \) with eigenvalue \( \leq \lambda \) and $\|u\|_{L^2}=1$, then for any \( C > 1 \),  
\[
\int_{\{x \in M : V(x) \geq C\lambda\}} |u|^2(x) \,dx \leq \frac{1}{C}.
\]  
\end{lem}
		\begin{proof}
				This is because
			\begin{align*}
				&\ \ \ \ C\l \int_{\{x \in M : V(x) \geq C\l\}} |u|^2(x) \, dx \leq \int_{\{x \in M : V(x) \geq C\l\}} V |u|^2(x) \, dx \\
				&\leq \int_M \left( \h^2 |\nabla u|^2(x) + V |u|^2(x) \right) \, dx \leq\l.
			\end{align*}
		\end{proof}
Let \(\lambda_k(\h)\) be the \(k\)-th eigenvalue of \(\Delta_\h = \h^2\Delta + V\).  
For any \(t > 0\), we show that for some \(\Lambda > 0\), the sum \(\sum_{\lambda_k(\h) \leq \frac{\Lambda}{t}} e^{-t\lambda_k(\h)}\) makes a significant contribution to the heat trace of \(e^{-t\Delta_\h}\).
			\begin{lem}\label{prop39}
			 Assume $V$ satisfies the doubling condition \eqref{double}. Then for any \(\epsilon,\delta > 0\), there exists a constant \(\Lambda = \Lambda(\epsilon,\delta) > 0\), independent of \((t, \h)\), such that
			\be\label{eq82} 
			\sum_{\l_k(\h) \geq \frac{\Lambda}{t}} e^{-t\l_k(\h)} < \epsilon \sum_{k}e^{-t(1-\delta)\l_k(\h)}.
			\ee
	\end{lem}
\begin{proof}
	This is because, for any $\Lambda>2$,
	\begin{align*}
			\sum_{\l_k(\h) \geq \frac{\Lambda}{t}} e^{-t\l_k(\h)} \leq e^{-\delta\Lambda}\sum_{\l_k(\h) \geq \frac{\Lambda}{t}} e^{-t(1-\delta)\l_k(\h)}\leq e^{-\delta\Lambda}\sum_{k} e^{-t(1-\delta)\l_k(\h)}.
	\end{align*}
\end{proof}

		Recall that $K_{\h}$ is the heat kernel associated with $\h^2\Delta + V$. We have:
		
		\begin{prop}\label{prop47}
		Assume $V$ satisfies the doubling condition \eqref{double}.	For any $\epsilon,\delta\in(0,\half)$,
			there exists $B = B(\epsilon,\delta) > 0$ such that for all $t,\h\in(0,1]$,
			\[
			\frac{\int_{\{x \in M : V(x) \geq \frac{B}{t}\}} K_{\h}(t,x,x)\,dx}{\Tr(e^{-t(1-\delta)(\h^2\Delta + V)})} \leq \epsilon .
			\]
		\end{prop}
		
		\begin{proof}
		 Let $\lambda_k(\h)$ be the $k$-th eigenvalue of $\h^2\Delta + V$, and let $u_k$ denote the corresponding unit eigenfunction. Then,
			\[
			K_{\h}(t,x,x) = \sum_k e^{-t\lambda_k(\h)} |u_k(x)|^2.
			\]
				Let $\Lambda=\Lambda(\ep,\delta)$ be determined in \Cref{prop39}. Set
			\[
			K^1_{\h}(t,x,x) = \sum_{\{k:\lambda_k(\h) \geq \frac{\Lambda}{t}\}} e^{-t\lambda_k(\h)} |u_k(x)|^2.
			\]
			By \Cref{prop39}, for any $B > 0$, we have
			\begin{equation}\label{eq88}
				\frac{\int_{\{x \in M : V(x) \geq \frac{B}{t}\}} K^1_{\h}(t,x,x)\,dx}{\Tr(e^{-t(1-\delta)(\h^2\Delta + V)})} \leq \frac{\int_M K^1_{\h}(t,x,x)\,dx}{\Tr(e^{-t(1-\delta)(\h^2\Delta + V)})} = \frac{\sum_{\{k:\lambda_k(\h) \geq \frac{\Lambda}{t}\}} e^{-t\lambda_k(\h)}}{\Tr(e^{-t(1-\delta)(\h^2\Delta + V)})} \leq \epsilon.
			\end{equation}
			Next, set
			\[
			K^2_{\h}(t,x,x) = \sum_{\{k:\lambda_k(\h) < \frac{\Lambda}{t}\}} e^{-t\lambda_k(\h)} |u_k(x)|^2.
			\]
			By \Cref{prop390}, setting $B = \epsilon^{-1} \Lambda$, we see that
			\begin{equation}\label{eq89}
				\frac{\int_{\{x \in M : V(x) \geq \frac{B}{t}\}} K^2_{\h}(t,x,x)\,dx}{\Tr(e^{-t(1-\delta)(\h^2\Delta + V)})} \leq \frac{\epsilon \sum_{\{k:\lambda_k(\h) < \frac{\Lambda}{t}\}} e^{-t\lambda_k(\h)}}{\Tr(e^{-t(1-\delta)(\h^2\Delta + V)})} \leq \epsilon.
			\end{equation}
			
			The proposition follows from \eqref{eq88} and \eqref{eq89}.
		\end{proof}
Now we proceed to prove asymptotic formulas \eqref{asymtr1} and \eqref{asymtr2} in \Cref{heatexpan}. Fix any \(\ep > 0\), and let \(A = A(\ep)\) be determined by \Cref{prop311}. Using \Cref{prop311} and \eqref{eq56}, we have
\begin{equation}\label{eq55}
\begin{aligned}
\liminf_{t \to 0} \frac{\Tr(e^{-t(\Delta + V)})}{(4\pi t)^{-\n} \int_M e^{-tV(x)}\,dx} 
&\geq (1 - \ep) \liminf_{t \to 0} \frac{\int_{\{x : V(x) \leq At^{-1}\}} K_{\h=1}(t,x,x)\,dx}{(4\pi t)^{-\n} \int_{\{x : V(x) \leq At^{-1}\}} e^{-tV(x)}\,dx} \\
&= 1 - \ep.
\end{aligned}
\end{equation}
By \Cref{cor31}, there exists \(\delta_0 \in (0, \frac{1}{2})\) such that
\begin{equation}\label{delta0}
\frac{\int_{M} e^{-t(1 - \delta_0)V(x)}\,dx}{\int_M e^{-tV(x)}\,dx} \leq 2.
\end{equation}
Next, let \(L = \max\{ A(\ep), B(\ep, \delta_0)\}\), where \(A(\ep)\) and \(B(\ep, \delta_0)\) are determined by \Cref{prop311} and \Cref{prop47}. Then:
\be\label{eq571}\ba
&\ \ \ \ 	\limsup_{t\to0}\frac{\Tr(e^{-t(\Delta+V)})}{(4\pi t)^{-\n}\int_M e^{-tV(x)}dx}\\&\leq\limsup_{t\to0}\frac{\int_{\{x:V(x)\leq Lt^{-1}\}}K_{\h=1}(t,x,x)dx+\ep\Tr(e^{-t(1-\delta_0)(\Delta+V)})}{(4\pi t)^{-\n}\int_{M} e^{-tV(x)}dx}\\
&\leq \limsup_{t\to0}\frac{\int_{\{x:V(x)\leq Lt^{-1}\}}K_{\h=1}(t,x,x)dx}{(4\pi t)^{-\n}\int_{M} e^{-tV(x)}dx}+2\ep\limsup_{t\to0}\frac{\Tr(e^{-t(1-\delta_0)(\Delta+V)})}{(4\pi t)^{-\n}\int_M e^{-t(1-\delta_0)V(x)}dx}\\
&\leq (1+\ep)\limsup_{t\to0}\frac{\int_{\{x:V(x)\leq Lt^{-1}\}}K_{\h=1}(t,x,x)dx}{(4\pi t)^{-\n}\int_{\{x:V(x)<Lt^{-1}\}} e^{-tV(x)}dx}+2\ep\limsup_{t\to0}\frac{\Tr(e^{-t(\Delta+V)})}{(4\pi t)^{-\n}\int_M e^{-tV(x)}dx}
\\&= 1+\ep+2\ep\limsup_{t\to0}\frac{\Tr(e^{-t(\Delta+V)})}{(4\pi t)^{-\n}\int_M e^{-tV(x)}dx},
\ea	\ee
where the first inequality follows from \Cref{prop47}, the second inequality follows form \eqref{delta0}, the third inequality follows from \Cref{prop311}, and the equality follows from \eqref{eq56}. 

From \eqref{eq571}, we conclude that
\begin{equation}\label{eq581}
\limsup_{t \to 0} \frac{\Tr(e^{-t(\Delta + V)})}{(4\pi t)^{-\n} \int_M e^{-tV(x)}\,dx} \leq \frac{1 + \ep}{1 - 2\ep}.
\end{equation}
Finally, combining \eqref{eq55} and \eqref{eq581} and letting \(\ep \to 0\), we establish \eqref{asymtr1}. 

The asymptotic behavior in \eqref{asymtr2} can be derived analogously. 

\appendix
\section{Comparison of Function Spaces}\label{appendix-A}
In this section, we compare several function spaces considered in classical literature on the Weyl law for Schr\"odinger operators on $\mathbb{R}^n$. We will show that the function spaces introduced in this paper are significantly larger than all the previously studied ones.

\subsection{${\mathcal{O}}_\beta' \subset \tilde{\mathcal{O}}_\beta$}\label{appendix a1}
In this subsection, we show that when $M = \R^n$, we have the inclusion $\mO_\beta' \subset \tmO_\beta,\b\in[0,\half)$. Recall that in $\R^n,$ we take $\tau_0=\sqrt{n}$.

It suffices to verify  \eqref{beta-regularity1}. Assume $V \in \mO_\beta'$, and recall that $\Omega_\lambda = \{x \in \R^n : V(x) < \lambda\}$. We may as well assume that $V\geq1$ a.e.
By \eqref{beta-regularity2}, 
\[\ba
&\quad \int_{B_{\tau_0}(x)} \int_{ S_r(z)\cap B_{\tau_0}(x)}|V(z) - V(y)|\,\dvol_{S_r(z)}(y)\, dz \\
&=\int_{B_{\tau_0(x)}}  \int_{ \{w\in S_r(0):z+w\in B_{\tau_0}(x)\}}|V(z) - V(z+w)|\,\dvol_{S_r(0)}(w)\, dz\\
&=\int_{S_{r(0)}}  \int_{ \{z\in B_{\tau_0}(x):z+w\in B_{\tau_0}(x)\}}|V(z) - V(z+w)|\,dz\,\dvol_{S_r(0)}(w)\\
&\leq \int_{S_r(0)} \eta(r)r^{2\beta}V^{1+\beta}(x)\,\dvol_{S_r(0)}(w)=\eta(r)r^{2\beta+n-1}V^{1+\beta}(x)|S_1(0)|.
\ea\]
 Thus, $\mO_\b'\subset\tmO_\b.$

\subsection{$\tmO_\b\subset \mO_\b$}
First, by the volume comparison theorem and the Vitali covering lemma (see \cite[$\S$ 1.3]{fanghua2003geometric}), there is a collection of balls $\{B_i\}_{i \in \mathbb{Z}}$ of radius $\tau_0$ that cover $M$, and each point $p \in M$ lies in at most $N$ of these balls, 
for some constant $N = N(\tau_0, R_0) > 0$.

Assume $V \in \tmO_\beta,\b\in[0,\half)$. We may as well assume that $V\geq1$ a.e.

Let $\{B_j\}_{j \in I} \subset \{B_i\}_{i \in \mathbb{Z}}$ be the collection of balls such that $|B_j \cap \Omega_\lambda| > 0$ for each $j \in I$. Then, by \eqref{beta-regularity}, we have, up to a set of measure zero,
\begin{equation} \label{Qj-inclusion}
\cup_{j \in I} B_j \subset \Omega_{C_V' \lambda}.
\end{equation}
Now, for $r \in (0, \tau_0)$, let $B_j^r \subset B_j$ denote the subset of $B_j$ consisting of points whose distance to the boundary $\partial B_j$ is at least $r$.

First, note that
\be\ba
&\quad\int_{B_j}\int_{ S_r(x)}  |V(x) - V(y)|\,\dvol_{S_r(x)}(y)\,dx\\
&=\int_{B_j^r}\int_{ S_r(x)}  |V(x) - V(y)|\,\dvol_{S_r(x)}(y)\,dx+\int_{B_j-B_j^r}\int_{ S_r(x)}  |V(x) - V(y)|\,\dvol_{S_r(x)}(y)\,dx\\
&=:I_1+I_2.
\ea\ee

By \eqref{beta-regularity1}, \eqref{Qj-inclusion} and volume comparison, we estimate
\be\ba \label{interior-int}
&\quad I_1\leq  \int_{B_j} \int_{ S_r(x)\cap B_j}|V(x) - V(y)|\dvol_{S_r(x)}(y) dx \\
&\leq C\eta(r)\, r^{n+2\beta-1} \l^{1+\beta}\leq C'\eta(r)\, r^{n+2\beta-1} \l^{1+\beta}|B_j|.
\ea\ee
where $|B_j|$ denotes the volume of $B_j$.

For $I_2$, by \eqref{beta-regularity} and volume comparison, we have if $\l\geq1$,
\be\ba \label{boundary-int}
&\quad I_2\leq\int_{B_j-B_j^r} r^{n-1}(\l+C_V'\l)dx\leq C r^n\l \\
&\leq  C' r^n\l |B_j|\leq C' r^{1-2\b}r^{n+2\b-1} \l^{1+\b} |B_j|.
\ea\ee

Recall that each point is at most covered by $N$ balls. Summing over all $j \in I$, and using \eqref{Qj-inclusion}--\eqref{boundary-int}, and \eqref{double}, we obtain, if $\l$ is large, 
\be\ba\label{verify-beta-regularity}
&\quad \int_{\Omega_\l}\int_{S_r(x)}  |V(x) - V(y)|\,\dvol_{S_r(x)}(y)\,dx \leq \sum_{j \in I} \int_{B_j}\int_{S_r(x)}  |V(x) - V(y)|\,\dvol_{S_r(x)}(y)\,dx \\
&\leq C N\big( \eta(r) + r^{1 - 2\beta} \big) r^{n + 2\beta-1} \lambda^{1 + \beta} \sigma(C_V' \lambda) \leq C''\big( \eta(r) + r^{1 - 2\beta} \big) r^{n + 2\beta-1} \lambda^{1 + \beta} \sigma(\lambda).
\ea\ee

Setting $\tilde\eta(r)= C'' \big( \eta(r) + r^{1 - 2\beta} \big)$, and noting that $\tilde\eta(r)\leq\tilde\eta\big(\l^{-1/3}\big)$ if $r\in(0,\l^{-1/3})$, we verify the condition \eqref{beta-regularity0} with $\mu(\l)=\l^{\frac{1}{6}}$. Hence we conclude that $\tmO_\beta \subset \mO_\beta$.

\subsection{$\mR_\b\subset\mO_\b$}
Let $V \in \mR_\b, \b\in[0,\half]$. Without loss of generality, we may assume that $V \geq 1$ a.e. It follows from \eqref{secondcond} and volume comparison that if $\l$ is large and $r\leq \tau_0$,
\[
\begin{aligned}
&\quad \int_{\Omega_\lambda - \Omega_{\sqrt{\lambda}}} \int_{S_r(x)} |V(x) - V(y)|\,\dvol_{S_r(x)}(y)\,dx \\
&\leq \int_{\Omega_\lambda - \Omega_{\sqrt{\lambda}}} \int_{S_r(x)} \lambda^{1 + \b} v(\sqrt{\lambda}) r^{2\b} \,\dvol_{S_r(x)}(y)\,dx \\
&\leq C v(\sqrt{\lambda})\, r^{n + 2\b-1} \lambda^{1 + \b} \sigma(\lambda),
\end{aligned}
\]
and similarly,
\[
\begin{aligned}
&\quad \int_{\Omega_{\sqrt{\lambda}}} \int_{S_r(x)} |V(x) - V(y)|\,\dvol_{S_r(x)}(y)\,dx \\
&\leq \int_{\Omega_{\sqrt{\lambda}}} \int_{S_r(x)} \sqrt{\lambda}^{1 + \b} v(1) r^{2\b} \,\dvol_{S_r(x)}(y)\,dx \\
&\leq C v(1)\, {\lambda}^{-\frac{1+\b}{2}} \, r^{n + 2\b-1} \lambda^{1 + \b} \sigma(\lambda).
\end{aligned}
\]

Therefore, setting ${\eta}(r) = C v(r^{-\half}) + C v(1) r^{\frac{1 + \b}{2}}$, the condition \eqref{beta-regularity0} is satisfied, and we conclude that $\mR_\b \subset \mO_\b$.
\subsection{More Function Spaces}\label{appendix a4}
\def\tmR{{\widetilde{\mathcal{S}}}}
For $a \in [0, \tfrac{1}{2})$, let $\tmR_a$ be the class of functions satisfying the same conditions as $\mathcal{R}_a$, except that \eqref{secondcond} is replaced by:
\begin{equation}\label{smooth-a}
V \in \mathrm{Lip}(M) \quad \text{and} \quad |\nabla V(x)| \leq C_V'' \max\{1, V(x)\}^{1+a} \quad \text{a.e.},
\end{equation}
for some constant $C_V'' > 1$. Here $\mathrm{Lip}(M)$ denotes the space of Lipchitz functions on $M$.

\begin{prop}
For any $\beta \in (a, \tfrac{1}{2})$, we have $\tmR_a \subset \mathcal{O}_\beta$.
\end{prop}

\begin{proof}
Let $V\in \tmR_a$. 
Fix $\beta \in (a, \tfrac{1}{2})$. For any set $U \subset M$ and $r > 0$, consider
\[
U^{+r} := \{ x \in M : d(x, U) < r \}.
\]
Let $\lambda > 1$ be sufficiently large. We claim that for any $r \in (0, \lambda^{-\beta})$,
\begin{equation}\label{osicillation-small}
\Omega_\lambda^{+r} \subset \Omega_{C_V'' \lambda}.
\end{equation}
Assuming the claim, we estimate using \eqref{smooth-a}:
\[
\begin{aligned}
&\quad \int_{\Omega_\lambda} \int_{S_r(x)} |V(x) - V(y)|\,\dvol_{S_r(x)}(y)\,dx \leq \int_{\Omega_\lambda} \int_{S_r(x)} r \cdot \sup_{z \in \Omega_\lambda^{+r}} |\nabla V(z)| \,\dvol_{S_r(x)}(y)\,dx \\
&\leq C \int_{\Omega_\lambda} \int_{S_r(x)} r \lambda^{1+a} \,\dvol_{S_r(x)}(y)\,dx \leq C' \lambda^{1+a} r^n \sigma(\lambda) \leq C' \lambda^{a - \beta} \lambda^{1 + \beta} r^{n - 1 + 2\beta} \sigma(\lambda).
\end{aligned}
\]
Hence the condition \eqref{beta-regularity0} is satisfied, and it remains to prove the claim \eqref{osicillation-small}.

Let $d := d(\partial \Omega_{C_V'' \lambda}, \partial \Omega_\lambda)$ and let $\gamma: [0, d] \to M$ be a unit-speed minimizing geodesic connecting a point on $\partial \Omega_\lambda$ to a point on $\partial \Omega_{C_V'' \lambda}$. Then $V(\gamma(s)) \in (\lambda, C_V'' \lambda)$ for all $s \in (0, d)$, otherwise it contradicts the minimality of $\gamma$. Using \eqref{smooth-a}, we get
\[
(C_V'' - 1) \lambda = |V(\gamma(0)) - V(\gamma(d))| \leq \int_0^d |\nabla V(\gamma(s))|\,ds \leq C_V'' \lambda^{1 + a} d,
\]
which implies
\[
d \geq \frac{C_V'' - 1}{C_V'' \lambda^a}.
\]
Thus, for large $\lambda$, any $r < \lambda^{-\beta}$ with $\beta > a$ satisfies $r < d$, and hence \eqref{osicillation-small} holds.
\end{proof}

Let $\widetilde{\mathcal{R}}_0$ be the class of functions satisfying the same conditions as $\mathcal{R}_a$, except that \eqref{secondcond} is replaced by the following: there exists an increasing function $\eta \in C([0, \tau_0))$ with $\eta(0) = 0$ such that for almost every $d(x,y)<\tau_0$,
\begin{equation}\label{smooth-0}
|V(x) - V(y)| \leq \eta\big(d(x, y)\big) \max\{1, |V(x)|\}.
\end{equation}

\begin{prop}
We have $\widetilde{\mathcal{R}}_0 \subset \mathcal{O}_0$.
\end{prop}

\begin{proof}
Let $V \in \tilde{\mathcal{R}}_0$, and fix any $\gamma \in (0, \tfrac{1}{2})$. Then, for any $r \in (0, \lambda^{-\gamma})$ and sufficiently large $\lambda$, we estimate
\[
\begin{aligned}
&\quad \int_{\Omega_\lambda} \int_{S_r(x)} |V(x) - V(y)|\,\dvol_{S_r(x)}(y)\,dx \\
&\leq \int_{\Omega_\lambda} \int_{S_r(x)} \eta(\lambda^{-\gamma}) \lambda \,\dvol_{S_r(x)}(y)\,dx \\
&\leq C \eta(\lambda^{-\gamma}) \lambda r^{n-1} \sigma(\lambda),
\end{aligned}
\]
which verifies \eqref{beta-regularity0} with $\beta = 0$. Hence $V \in \mathcal{O}_0$.
\end{proof}
\subsection{Compare with classical results on $\R^n$}\label{classical}
Assuming that $V \to \infty$ and satisfies the doubling condition, \Cref{weyl1} extends several classical results to general noncompact manifolds with bounded geometry, under assumptions that are strictly and substantially weaker. For more details, see the discussion below.

The spaces $\mO_\beta'$, $\beta \in [0, \tfrac{1}{2})$, and $\mR_a$, $a \in [0, \tfrac{1}{2}]$, were studied in \cite{rozenbljum1974asymptotics}. Note that $\mR_\b$ and $\mO_\beta'$ are not contained in each other, but we have shown that $\mR_\b \subset \mO_\b$ and $\mO_\beta' \subset \tmO_\beta\subset\mO_\b$. As a result, $\mO_\beta'$ is strictly contained in $\mO_\beta,\b\in[0,\half)$.

The space ${\tmR}_a$, $a \in [0, \tfrac{1}{2})$, was considered by Tachizawa \cite[Theorem 4.3]{tachizawa1992eigenvalue} and Feigin \cite{feigin1976asymptotic}. However, Tachizawa's method applies only to $\mathbb{R}^n$ with $n \geq 3$, whereas our \Cref{weyl1} imposes no dimensional restriction. Feigin studied the asymptotic distribution of eigenvalues of pseudo-differential operators, including the Schrödinger operator $\Delta + V$, but required extra conditions on higher-order derivatives of $V$.

The space $\widetilde{\mR}_0$ was considered by Fleckinger \cite{fleckinger1981estimate}, though with additional assumptions imposed on the potential function. Specifically, consider a partition of $\mathbb{R}^n$ into a family of disjoint open cubes $\{Q_i\}_{i \in \mathbb{Z}}$ of fixed side length $\eta > 0$. Consider
\[
I := \left\{ i \in \mathbb{Z} : \bar{Q}_i \subset \Omega_\lambda \right\}, \quad
J := \left\{ i \in \mathbb{Z} : \bar{Q}_i \cap \Omega_\lambda \neq \varnothing \right\},
\]
where $\Omega_\lambda = \{ x \in \mathbb{R}^n : V(x) \leq \lambda \}$. Then the following condition is required by Fleckinger:
\[
\lim_{\eta \to 0} \frac{\#(J \setminus I)}{\# J} = 0 \quad \text{when $\l$ is large }.
\]
\def\subs{1}
\if\subs1
\def\tv{{\tilde{v}}}
\def\ttv{{\tilde{\tv}}}
\subsection{Removing the Bounded Geometry Assumption}\label{removing}

We briefly discuss how to relax the bounded geometry assumption. Curvature bounds appear only in the following estimates:

\begin{itemize}
  \item \textbf{\eqref{rh}}:  Uses the curvature and its first covariant derivative to estimate \( \Delta G^{-1/4} \).
    \item \textbf{\eqref{tilderhk1}}: Uses curvature bounds to estimate \( \Delta d^2(x,y) \) (differentiation in the \( y \)-variable).
  \item \textbf{\eqref{heat estimate}}: Requires curvature bounds to obtain the heat kernel estimate.
  \item \textbf{\eqref{Gaussint}}: Uses Ricci curvature lower bounds for volume comparison arguments.
\end{itemize}
We now discuss how to extend the results beyond the bounded geometry setting:
\begin{itemize}
  \item \textbf{Bounded curvature, but injectivity radius degenerates.}  
  Let \( \mO_\beta^{\mathrm{inj}} \subset \mO_\beta \), where \( \beta \in [0, \tfrac{1}{2}] \), consist of functions \( V \geq 0 \) such that for large \( \lambda \), the injectivity radius at any point \( x \notin \Omega_\lambda \) is bounded below by \( V(x)^{-1/2} \mu(x) \), with \( \mu \) as in \eqref{beta-regularity0}. Our arguments remain valid for functions in $\mO_\beta^{\mathrm{inj}}$.
  
  \item \textbf{Unbounded curvature.}  
  Suppose the curvature as well as its first covariant derivative is not uniformly bounded, but satisfies an upper bound of the form \( f(V(x)) \), where \( f : [0, \infty) \to [0, \infty) \) is continuous and satisfies \( \lim_{\lambda \to \infty} f(\lambda) = \infty \). Then the estimates similar to \eqref{tilderhk1} and \eqref{Gaussint} remain valid within geodesic balls of radius \( V(x)^{-1/2} \mu(x) \) for some suitable $f$, using local comparison geometry. The heat kernel estimate \eqref{heat estimate} also holds in such balls for small \( t \), via rescaling and finite propagation speed arguments.
\end{itemize}

The arguments in this paper can also be extended to Schr\"odinger operators on vector bundles $E \to M$ (and hence to magnetic Schr\"odinger operators on $\mathbb{R}^n$), provided that the curvature of $E$ and its first covariant derivative are controlled by the potential $V$ as described above.
\subsection{Beyond Classical KHL Tauberian theorem}\label{non-regular}
In this subsection, we will show that there exists an increasing function $\nu$ on $[0,\infty)$ such that $\nu$ satisfies the doubling condition, but 
\[
\int_0^\infty e^{-t\l} \, d\nu(\l)
\]
is not asymptotically regularly varying.

Consider the increasing function
\[
\nu(\l):=
\begin{cases}
0, & \l \in [0,\half); \\[0.5ex]
2^{k-1}, & \l \in (2^{k-1}, 2^k],\quad \text{$k \geq 0$ even}; \\[0.5ex]
\sqrt{2}\,2^{k-1}, & \l \in (2^{k-1}, 2^k],\quad \text{$k \geq 0$ odd}.
\end{cases}
\]
One can easily verify that
\[
\frac{\nu(2^k)}{\nu(2^{k-1})} =
\begin{cases}
2\sqrt{2}, & \text{if $k$ is odd}; \\
\sqrt{2}, & \text{if $k$ is even};
\end{cases}
\]
and that
\[
\nu(2\l) \leq 2\sqrt{2}\,\nu(\l), \quad \l \geq 1.
\]
Thus, $\nu$ satisfies the doubling condition, but the limit
\[
\lim_{\l \to \infty} \frac{\nu(2\l)}{\nu(\l)}
\]
does not exist.

Now, suppose
\[
\int_0^\infty e^{-t\l} \, d\nu(\l)
\]
is asymptotically regularly varying as $t \to 0^+$. Then by the classical KHL Tauberian theorem, $\nu(\l)$ must be asymptotically regularly varying as $\l \to \infty$, which implies that
\[
\lim_{\l \to \infty} \frac{\nu(2\l)}{\nu(\l)}
\]
exists—a contradiction.
\fi
\section{On Sharpness of $\beta$-oscillation conditions}\label{appendix-B}
The following construction is from \cite[§6]{rozenbljum1974asymptotics}. 
In $\mathbb{R}^n$, $n \geq 2$, consider the set
\[
U = \left\{ x = (x', s) \in \mathbb{R}^{n-1} \times \mathbb{R} : 1 < s < \infty,\ |x'| < s^{-\theta} \right\}.
\]
Consider
\[
V(x) :=
\begin{cases}
|x|^{\kappa_2}, & x \in U, \\
|x|^{\kappa_1}, & x \in \mathbb{R}^n \setminus U,
\end{cases}
\]
where the parameters satisfy
\[
\frac{1}{n-1} > \theta > \frac{\kappa_1}{2}, \quad \frac{1 - \theta(n-1)}{\kappa_2} > \frac{n}{\kappa_1}, \quad \kappa_1 < 1.
\]

It is shown in \cite[§6]{rozenbljum1974asymptotics} that this potential $V$ satisfies the doubling condition \eqref{double} and the condition \eqref{beta-regularity2}, but fails to satisfy the  condition \eqref{beta-regularity}. Moreover, the classical Weyl law \eqref{heatWeyl1} for $\Delta + V$ does not hold.

We further show that $V$ also fails oscillation condition \eqref{beta-regularity0}.

Recall that $\Omega_\lambda := \{ x \in \mathbb{R}^n : V(x) \leq \lambda \}$ and $\sigma(\lambda) := |\Omega_\lambda|$. For the potential $V$ described above, a direct computation shows that 
\[
\sigma(\lambda) \approx \lambda^{\frac{1 - \theta(n-1)}{\kappa_2}}.
\]
Here, for two functions $f, g$ on $[0,\infty)$, we write $f \approx g$ if there exists a constant $C > 1$ such that for all sufficiently large $\lambda$,
\[
C^{-1} g(\lambda) \leq f(\lambda) \leq C g(\lambda).
\]

Now consider the subset
\[
O_\lambda := \left\{ x \in U : \lambda^{\frac{1}{\kappa_1}} \leq |x| \leq \lambda^{\frac{1}{\kappa_2}} \right\}.
\]
Let $r = \lambda^{-1/2}$. Since $r \gg \lambda^{-\theta/\kappa_1}$ for large $\lambda$, we obtain for large $\l,$
\[
\int_{O_\lambda} \int_{S_r(x)} |V(x) - V(y)|\, \dvol_{S_r(x)}(y)\, dx 
\approx r^{n-1} \int_{\lambda^{1/\kappa_1}}^{\lambda^{1/\kappa_2}} s^{-\theta(n-1)} s^{\kappa_1}\, ds 
\approx r^{n-1} \lambda^{\frac{1 - \theta(n-1) + \kappa_1}{\kappa_2}}.
\]

Therefore, for any $\beta \in [0, \tfrac{1}{2}]$,
\[
\frac{ \displaystyle \int_{O_\lambda} \int_{S_r(x)} |V(x) - V(y)|\, \dvol_{S_r(x)}(y)\, dx }{ r^{n-1 + 2\beta} \lambda^{1+\beta} \sigma(\lambda) } 
\approx \lambda^{\frac{\kappa_1}{\kappa_2} - 1},
\]
which diverges as $\lambda \to \infty$ since $\kappa_1 > \kappa_2$. 
This shows that $V$ fails to satisfy the $\beta$-oscillation condition \eqref{beta-regularity0}.

		\bibliography{lib}
		\bibliographystyle{plain}
	\end{document}